\documentclass[11pt,a4paper]{amsart}
\usepackage{a4wide}
\usepackage{amsmath, amssymb, amsfonts,enumerate}
\usepackage[all]{xy}
\usepackage{amscd}
\usepackage{comment}
\usepackage{mathtools}
\usepackage{tikz} 
\usepackage{tikz-cd} 

 \newtheorem{theorem}{Theorem}[section]
\newtheorem{lemma}[theorem]{Lemma}
\newtheorem{corollary}[theorem]{Corollary}
\newtheorem{proposition}[theorem]{Proposition}
\newtheorem{claim}{Claim} 
 
 \theoremstyle{definition}
 \newtheorem{definition}[theorem]{Definition}
 \newtheorem{remark}[theorem]{Remark}

 \newtheorem{example}[theorem]{Example}

\numberwithin{equation}{section}
\newcommand {\N}{\mathbb{N}} %% positive integers
\newcommand {\Z}{\mathbb{Z}} %% integers
\newcommand {\R}{\mathbb{R}} %% reals
\newcommand {\C}{\mathbb{C}} %% complex

\newcommand{\Proj}{\mathbb{P}}

\newcommand{\EE}{\mathcal{E}}

\newcommand{\A}{\mathbb{A}} %% affine space

\DeclareMathOperator{\Per}{Per}

\DeclareMathOperator{\Fix}{Fix}

\DeclareMathOperator{\im}{Im}

\DeclareMathOperator{\Id}{Id}
\DeclareMathOperator{\Stab}{Stab}

\DeclareMathOperator{\Spec}{Spec}

\DeclareMathOperator{\NW}{NW}
\DeclareMathOperator{\Rec}{R}
\DeclareMathOperator{\Crec}{CR}
 \DeclareMathOperator{\CPSFT}{CPSFT} % countably-proconstructible sft 
 \DeclareMathOperator{\CPS}{CPS} % countably-proconstructible sofic  

\begin{document}
\title[Endomorphisms of algebraic sofic shifts]{Invariant sets and nilpotency of endomorphisms of algebraic sofic shifts}
\author[T. Ceccherini-Silberstein]{Tullio Ceccherini-Silberstein}
\address{Dipartimento di Ingegneria, Universit\`a del Sannio, C.so
Garibaldi 107, 82100 Benevento, Italy}
\email{tullio.cs@sbai.uniroma1.it}
\author[M. Coornaert]{Michel Coornaert}
\address{Universit\'e de Strasbourg, CNRS, IRMA UMR 7501, F-67000 Strasbourg, France}
\email{michel.coornaert@math.unistra.fr}
\author[X.K.Phung]{Xuan Kien Phung}
\address{D\'epartement d'informatique et de recherche op\'erationnelle,  Universit\'e de Montréal, Montr\'eal, Québec, H3T 1J4, Canada}
\email{phungxuankien1@gmail.com}
\subjclass[2010]{37B15, 14A10, 14A15, 37B10, 68Q80}
\keywords{Algebraic variety, algebraic cellular automaton,  algebraic sofic subshift, space-time inverse system, 
nilpotency, limit set, infinite alphabet symbolic system}
\begin{abstract}
Let $G$ be a  group  and let $V$ be an algebraic variety over an algebraically closed field $K$. 
Let $A$ denote the set of $K$-points of $V$. 
We introduce algebraic sofic subshifts $\Sigma \subset A^G$ and study
endomorphisms $\tau \colon \Sigma \to \Sigma$. 
We generalize several results for dynamical invariant sets and nilpotency of $\tau$ 
that are well known for finite alphabet cellular automata. 
Under mild assumptions, we prove that $\tau$ is nilpotent if and only if its limit set, 
i.e., the intersection of the images of its iterates, is a singleton. 
If moreover $G$ is infinite, finitely generated and $\Sigma$ is topologically mixing, 
we show that $\tau$ is nilpotent if and only if its limit set consists of periodic configurations and 
has a finite set of alphabet values. 
\end{abstract}
\date{\today}
\maketitle

\setcounter{tocdepth}{1}
\tableofcontents
% SECTION 1
\section{Introduction} 
The main goal of the present paper is to extend and generalize  
well known results about limit sets and nilpotency of classical cellular automata, 
i.e., cellular automata with finite alphabets, to the setting of algebraic cellular automata over algebraic sofic subshifts, 
where the alphabet is the set of rational points of an algebraic variety. 
\par 
Since the pioneering work of John von Neumann in the 1940s~\cite{neumann-book},  
the mathematical theory of cellular automata has led to very interesting questions,
with deep connections to areas such as theoretical computer science, decidability, 
dynamical systems, ergodic theory, harmonic analysis, and geometric group theory.
In his empirical classification of the long-term behaviour of classical cellular automata, 
Wolfram~\cite{wolfram-univ-1984} introduced the notion of a limit set. 
For classical cellular automata, properties of limit sets and their relations with various notions of nilpotency were subsequently investigated by several authors 
(see in particular~\cite{milnor-ca-1988}, \cite{culik-limit-sets-1989}, \cite{kari-nilpotency-1992}, \cite{guillon-richard-2008}, \cite{salo-nil-2017}). In particular, Aanderaa and Lewis \cite{A-L} and, independently, Kari \cite{kari-nilpotency-1992} 
proved undecidability of nilpotency for classical cellular automata over $\Z$: this undecidability result
constitutes one of the most influential results in the theory of cellular automata and one of the main motivations for the study of
nilpotency in the symbolic dynamics setting.
In general, these properties of limit sets become false when the alphabet is allowed to be infinite. 
A major problem arising when working with infinite alphabets 
is that images of subshifts of finite type may fail to be closed (e.g.~\cite[Example~3.3.3]{book}).  
Nevertheless, infinite alphabet subshifts and their dynamics are not only intrinsically interesting but also fundamental to the study 
of smooth dynamical systems  
(cf. e.g.~\cite{boyle-buzzi-gomez-2006}, \cite[Ch.~7]{kitchens-book}, \cite{sarig-1999}, \cite{sarig-2013}, \cite{lima-sarig-2019} and the references therein).  
\par 
After Gromov \cite{gromov-esav}, the study of injectivity and surjectivity of algebraic cellular automata was  
pursued 
in~\cite{csc-algebraic}, \cite{csc-cat}, \cite{ccp-2019}, \cite{ccp-goe-2020}, \cite{phung-2018} 
to obtain generalizations of the Ax-Grothendieck theorem 
\cite{ax-injective}, \cite[Proposition~10.4.11]{grothendieck-4-3}  
and of the Moore-Myhill Garden of Eden theorem \cite{moore},~\cite{myhill}. 
Nilpotency is in the opposite direction since a nilpotent map is never~injective 
nor surjective when the underlying set has at least two
elements. 
\par
To state our results, let us first introduce some terminology and notation. 
Let $f \colon X \to X$ be a map from a set $X$ into itself.
Given an integer $n \geq 1$, the $n$-th iterate of $f$ is the map $f^n \colon X \to X$
defined by $f^n \coloneqq f \circ f \circ \cdots \circ f$ ($n$ times). 
The sets $f^n(X)$, $n \geq 1$, form a decreasing sequence of subsets of $X$. 
The \emph{limit set} $\Omega(f) \coloneqq \bigcap_{n \geq 1} f^n(X)$ of $f$ is the set of points  
that occur after iterating $f$  arbitrarily many times. 
\par 
Observe that  $f(\Omega(f)) \subset  \Omega(f)$. 
The inclusion may be strict
and equality holds if and only if every $x \in \Omega(f)$ admits a \emph{backward orbit}, 
i.e., a sequence $(x_i)_{i \geq 0}$ of points of $X$ such that 
$x_0 = x$ and $f(x_{i + 1}) = x_i$ for all $i \geq 0$. 
Clearly, $f$ is surjective if and only if $\Omega(f) = X$. 
Note also that $\Per(f) \coloneqq \bigcup_{n \geq 1} \{ x\in X \colon f^n(x)=x \} \subset \Omega(f)$
and that $\Omega(f^n) = \Omega(f)$ for every $n \geq 1$. 
The map  $f$ is \emph{stable} if $f^{n+1}(X)=f^n(X)$ for some $n \geq 1$. 
If $f$ is stable then $\Omega(f) \not= \varnothing$ unless $X = \varnothing$.
Clearly, $f$ is stable whenever $X$ is finite. 
If $X$ is infinite, there always exist maps $f \colon X \to X$ with $\Omega(f) = \varnothing$ (cf.~Lemma~\ref{l:non-nil-limit-one-pt}). 
\par
Assume that $X$ is a topological space and $f \colon X \to X$ is a continuous map. 
One says that $x \in X$ is a \emph{recurrent} (resp.~\emph{non-wandering}) point of $f$ if 
for every neighborhood $U$ of $x$, there exists $n \geq 1$ such that $f^n(x) \in U$ 
(resp.~$f^n(U)$ meets $U$).
Let $\Rec(f)$ (resp.~$\NW(f)$) denote the set of recurrent (resp.~non-wandering) points of $f$.
It is immediate that $\Per(f) \subset \Rec(f) \subset \NW(f)$ and that $\NW(f)$ is a closed subset of $X$.
In general, neither $\Per(f)$, nor $\Rec(f)$, nor $\Omega(f)$ are closed in $X$ 
(see Example~\ref{ex:limit-set-not-closed}).
\par
Suppose now that $X$ is a uniform space and $f \colon X \to X$ is a uniformly continuous map.
One says that a point $x \in X$ is \emph{chain-recurrent} if for every entourage $E$ of $X$ there exist an integer $n \geq 1$ and
a sequence of points $x_0,x_1,\dots,x_n \in X$ such that $x = x_0 = x_n$ and $(f(x_i),x_{i + 1}) \in E$ for all $0 \leq i \leq n - 1$.
We shall denote by  $\Crec(f)$   the set of chain-recurrent  points of $f$.
Observe that  $\Crec(f)$ is always closed in $X$.
\par
Let $G$ be a group and let  $A$ be a set, called the \emph{alphabet}.
The set $A^G  \coloneqq   \{x  \colon G \to A\}$, consisting of all maps from $G$ to $A$,  
is called the set of \emph{configurations} over the group $G$ and the alphabet $A$.
We equip $A^G = \prod_{g \in G} A$ with its \emph{prodiscrete uniform  structure}, i.e., 
the product uniform structure obtained by taking the discrete uniform structure on each factor $A$ of $A^G$.
Note that   $A^G$ is a totally disconnected Hausdorff space and that $A^G$ is compact if and only if $A$ is finite.
The \emph{shift action} of the group $G$  on $A^G$ 
is the action defined by $(g,x) \mapsto g x$, where $gx(h) \coloneqq  x(g^{-1}h)$ for all 
$g,h \in G$ and $x \in A^G$.
This action is uniformly continuous with respect to the prodiscrete uniform structure. 
\par 
For a subgroup $H \subset G$, define $\Fix(H) \coloneqq \{x \in A^G \colon hx = x \mbox{ for all } h \in H\}$. 
Then $\Fix(G)$ is the set of constant configurations while $\Fix(\{1_G\}) = A^G$. 
A configuration $x \in A^G$ is said to be \emph{periodic} 
if its $G$-orbit is finite, i.e., 
there is a finite index subgroup $H$ of $G$ such that $x \in \Fix(H)$. 
\par
A  $G$-invariant subset  $\Sigma \subset A^G$ is called a \emph{subshift} of $A^G$.
Note that we do not require closedness  in $A^G$ in our definition of a subshift.
\par
Given a finite subset   $D \subset G$ and 
a (finite or infinite) subset  $P \subset A^D$, 
the set 
\begin{equation}
\label{e:sft} 
\Sigma(D,P) \coloneqq \{ x \in A^G \colon (g^{-1}  x)\vert_{D} \in P \text{ for all } g \in G\}
\end{equation}
is a closed subshift of $A^G$
(here $(g^{-1}x)\vert_D \in A^D$ denotes the restriction of the configuration $g^{-1}x$ to  $D$). 
One says that $ \Sigma(D,P)$ is the 
\emph{subshift of finite type} associated with $(D,P)$ and that $D$ is a
\emph{defining memory set} for~$\Sigma$. 
\par
 Let $B$ be another alphabet set. 
 A map $\tau \colon B^G  \to A^G$ is called a \emph{cellular automaton} 
 if there exist a finite subset $M \subset G$ 
and a map $\mu \colon  B^M \to A$ such that 
\begin{equation} 
\label{e;local-property}
\tau(x)(g) = \mu( (g^{-1}x)\vert_M)  \quad  \text{for all } x \in B^G \text{ and } g \in G.
\end{equation}
Such a set $M$ is then called a \emph{memory set} and $\mu$ is called a \emph{local defining map} for $\tau$. 
It is clear from the definition that every cellular automaton 
$\tau \colon B^G \to A^G$ is uniformly continuous and $G$-equivariant 
(see~\cite{book}). 
\par 
More generally, if $\Sigma_1 \subset B^G$ and $\Sigma_2 \subset A^G$ are subshifts, 
a map $\tau \colon \Sigma_1 \to \Sigma_2$ is a \emph{cellular automaton} if it can be extended to 
a cellular automaton $B^G \to A^G$. 
\par 
Suppose  now that $U, V$ are algebraic varieties (resp. algebraic groups) over a field $K$, 
and let $A \coloneqq V(K)$, $B \coloneqq U(K)$ denote the sets of $K$-points of $V$ and $U$,  
i.e., the set consisting of all $K$-scheme morphisms $\Spec(K) \to V$ and $\Spec(K) \to U$, respectively.
See \cite[Appendix~A]{ccp-2019}, \cite[Section~2]{ccp-goe-2020} for basic definitions and properties of algebraic varieties.  
The following definition was introduced in the case $U=V$ in \cite[Definition~1.1]{ccp-2019} and \cite{phung-2018} after Gromov~\cite{gromov-esav}. 
A cellular automaton $\tau \colon B^G \to A^G$ is an \emph{algebraic (resp. algebraic group) cellular automaton}  
if $\tau$ admits a memory set $M$ 
with local defining map $\mu \colon B^M \to A$ 
 induced by some algebraic morphism (resp. homomorphism of algebraic groups)  
$f \colon U^M \to V$
(here $U^M$ denotes the $K$-fibered product of a family of copies of $U$ indexed by $M$). 
More generally, given subshifts $\Sigma_1 \subset B^G$ and $\Sigma_2 \subset A^G$,
a map $\tau \colon \Sigma_1 \to \Sigma_2$ is called an \emph{algebraic} (resp.~\emph{algebraic group}) 
\emph{cellular automaton}
if it is the restriction of some algebraic (resp.~algebraic group) cellular automaton $\tilde{\tau} \colon B^G \to A^G$  
(see Section~\ref{s:extension} for an example). 
\par
 Every cellular automaton with finite alphabet $A$ is an algebraic cellular automaton over any field $K$
(see the remarks after \cite[Definition~1]{ccp-2019}). Indeed, it suffices to embed $A$ as a subset
of $K$ and then observe that, if $M$ is a finite set, any map
$\mu \colon A^M \to A$ is the restriction of some polynomial map $P \colon  K^M \to K$ 
(which can be made explicit by using Lagrange interpolation formula). 
Similarly, any linear cellular automaton (cf.\ \cite[Chapter 8]{book}, \cite{JPAA}) 
is an algebraic cellular automaton: if $A$ is a finite-dimensional vector space over a field $K$, 
and $M$ is a finite set, then any linear map $\mu \colon A^M \to A$ is clearly a polynomial.

\begin{definition}
\label{d:alg-sft}
One says that $\Sigma \subset A^G$
is an \emph{algebraic (resp. algebraic group) subshift of finite type}
if there exist a finite subset $D \subset G$ and an algebraic subvariety (resp.\ algebraic subgroup) $W \subset V^D$ such that,
with the notation introduced in \eqref{e:sft},
one has $\Sigma = \Sigma(D,W(K))$. 
\end{definition}

By analogy with the definition of sofic subshifts in the classical setting, we define
algebraic sofic subshifts and algebraic group sofic subshifts as follows 
(see Definition~\ref{d:cpsft} for more general notions). 

\begin{definition} 
\label{d:alg-sofic-shift} 
Let $G$ be a group and let $V$ be an algebraic variety (resp. algebraic group) over a field $K$. 
Let $A \coloneqq V(K)$. 
A subset $\Sigma \subset A^G$  is called an \emph{algebraic (resp. algebraic  group) sofic subshift} 
if it is the image of an algebraic (resp. algebraic group) subshift of finite type $\Sigma' \subset B^G$, 
where $B=U(K)$ and $U$ is a $K$-algebraic variety (resp. $K$-algebraic group),
under an algebraic (resp. algebraic group) cellular automaton $\tau' \colon B^G \to A^G$.  
\end{definition} 

Every algebraic sofic subshift $\Sigma \subset A^G$ is indeed a subshift 
but it may fail to be closed in $A^G$ (cf.~Exemple~\ref{ex:limit-set-not-closed}). 
However, it turns out that, under suitable  natural conditions (see (H1), (H2), (H3) below), 
all algebraic sofic subshifts $\Sigma \subset A^G$ are closed in $A^G$ (cf.~Corollary~\ref{c:sofic-closed}). 
Moreover, we establish a fundamental characterization of algebraic subshifts of finite type by the descending chain property 
(cf. Theorem~\ref{t:noetheiran-alg-sft}). 
\par 
With the notation as in Definition \ref{d:alg-sofic-shift},
we shall investigate in this paper various dynamical aspects of an algebraic cellular automaton 
$\tau \colon \Sigma \to \Sigma$ over an algebraic sofic subshift $\Sigma \subset A^G$ satisfying one of the following hypotheses 
(with the same notations throughout the paper): 
\begin{enumerate}[(H1)]
\item
$K$ is an uncountable algebraically closed field (e.g.~$K = \C$);
\item
$K$ is algebraically closed and $U, V$ are complete (e.g.~projective) algebraic varieties over $K$;  
\item
$K$ is algebraically closed, $V$ is an algebraic group over $K$, $\Sigma \subset A^G$ is an algebraic group sofic subshift,
and $\tau \colon \Sigma \to \Sigma$ is an algebraic group cellular automaton. 
\end{enumerate}
\par 
We shall establish the following result.  
\begin{theorem}
\label{t:limit-set-not-empty}
Let $G$ be a group and let $V$ be an algebraic variety over a field $K$. 
Let $A \coloneqq V(K)$ and let $\Sigma \subset A^G$ be an algebraic sofic subshift. 
Let $\tau \colon \Sigma \to \Sigma$ be an algebraic cellular automaton and 
assume that one of the conditions $(\mathrm{H1})$, $(\mathrm{H2})$, $(\mathrm{H3})$ is satisfied. 
Then the following hold:
\begin{enumerate}[\rm (i)]
\item
 $\Omega(\tau)$ is a closed subshift of $A^G$;
  \item
 $\tau(\Omega(\tau)) = \Omega(\tau)$; 
 \item
 $\Per(\tau) \subset \Rec(\tau) \subset \NW(\tau) \subset \Crec(\tau)  \subset \Omega(\tau)$; 
  \item 
if $(\mathrm{H2})$ or $(\mathrm{H3})$ is satisfied and  
$\Omega(\tau)$ is a subshift of finite type, then $\tau$ is stable;  
\item 
for every subgroup $H \subset G$, 
if $\Sigma \cap \Fix(H) \neq \varnothing$ then $\Omega(\tau) \cap \Fix(H) \neq \varnothing$. 
\end{enumerate}
\end{theorem} 
\par
See \cite[Theorem 1.5]{JPAA} for a linear version of the above
theorem.
\par
One  says that a map $f \colon X \to X$ from a set $X$ into itself  is \emph{nilpotent}
if there exist  a constant map  $c \colon X \to X$ and an integer $n_0 \geq 1$ such that
$f^{n_0} = c$.
This implies $f^n = c$ for all $n \geq n_0$. 
Such  a constant map $c$ is then  unique and
we say that the unique  point  $x_0 \in X$ such that $c(x) = x_0$ for all $x \in X$ is the \emph{terminal point} of $f$.
The terminal point of a nilpotent map is its unique fixed point. 
\par 
Observe that if $f \colon X \to X$ is nilpotent with terminal point $x_0$ then $\Omega(f) = \{x_0\}$ is a singleton.
The converse is not true in general.
Actually, as soon as the set $X$ is infinite,
there exist non-nilpotent  maps $f \colon X \to X$ whose limit set  is reduced to a single point (cf.~Lemma~\ref{l:non-nil-limit-one-pt}).  
However, in the algebraic setting, we obtain the following result. 
\begin{theorem}
\label{t:char-nilpotent-alg-ca}
If we keep the same notation and hypotheses as in Theorem~\ref{t:limit-set-not-empty}, 
then the following conditions are equivalent: 
\begin{enumerate}[\rm (a)]
\item
$\tau$ is nilpotent;
\item
the limit set $\Omega(\tau)$ is reduced to a single configuration. 
\end{enumerate}
\end{theorem}
The analog of Theorem~\ref{t:char-nilpotent-alg-ca} for classical cellular automata follows from \cite[Theorem~3.5]{culik-limit-sets-1989}.
On the other hand, Theorem~\ref{t:char-nilpotent-alg-ca} can be seen as a generalization of   
an interesting and nontrivial property of endomorphisms of algebraic varieties (by taking $G = \{1_G\}$). 
\par
Both Theorem~\ref{t:limit-set-not-empty} and Theorem~\ref{t:char-nilpotent-alg-ca} become false if we remove the hypothesis that the ground field $K$ is algebraically closed
(see Example~\ref{ex:limit-set-not-closed}, Example~\ref{ex:alg-ca-empty-ls}, and Example~\ref{ex:alg-not-nil-ca-ls-single}).
To illustrate the significance of our results, note that 
for a group $G$ and a finite set $A$, the space $A^G$ is compact by Tychonoff's theorem.
Consequently, if $\Sigma \subset A^G$ is a closed subshift and $\tau \colon A^G \to A^G$ is a cellular automaton,
then $\tau^n(\Sigma)$ is closed in $A^G$ for every $n \geq 1$ and 
it follows that $\Omega(\tau)$ is a closed subshift of $A^G$. 
A standard compactness argument shows also that $\Omega(\tau) \not= \varnothing$ if $\Sigma \neq \varnothing$ 
(cf.~\cite{culik-limit-sets-1989}). 
When $A$ is infinite, $A^G$ is no longer compact
and, given a closed subshift $\Sigma \subset A^G$, the limit set of a cellular automaton $\tau \colon \Sigma \to \Sigma$ 
is, in general, no longer closed in $A^G$ 
(cf.~Example~\ref{ex:limit-set-not-closed}). 
Also, when $A$ is infinite, it may happen that $\Omega(\tau) = \varnothing$ 
while $\Sigma \neq \varnothing$ or even $\tau(\Omega(\tau)) \subsetneqq \Omega(\tau)$ (cf.~Proposition~\ref{p:ca-non-nil-limit-pt} and Example~\ref{ex:alg-ca-empty-ls}). 
\par  
A self-map $f \colon X \to X$ on a set $X$ is said to be \emph{pointwise nilpotent} 
if there exists  a point  $x_0 \in X$ such that for every $x \in X$, there exists an integer $n_0 \geq 1$ 
such that $f^n(x)=x_0$ for all $n \geq n_0$. 
\par 
Consider a group $G$ with the following property:   
for every finite alphabet $A$, 
any cellular automaton $\tau \colon A^G \to A^G$ with $\Omega(\tau)$ finite is nilpotent. 
Such a group  $G$ cannot be finite.
Indeed, for $G$ finite and   $A \coloneqq \{0,1\}$,
 the identity cellular automaton map $\tau \colon A^G \to A^G$
 has a finite limit set $\Omega(\tau) = A^G$
without being  nilpotent. 
By \cite[Corollary 4]{guillon-richard-2008} or
\cite{culik-limit-sets-1989}, we know that $G=\Z$ satisfies the above
property. 
In Theorem~\ref{t:finit-limit-set}, we show that actually it is satisfied by all infinite groups.  
\par  
More generally, we obtain the following various characterizations of nilpotent algebraic cellular automata. 

\begin{theorem}
\label{t:char-nilpotent-finite-alg-ca}
Let $G$ be an infinite group and le $V$ be an algebraic variety over a field $K$. 
Let $A=V(K)$ and let $\Sigma \subset A^G$ be a nonempty topologically mixing algebraic sofic subshift (e.g. $A^G$ for $A \not= \varnothing$).  
Let  $\tau \colon \Sigma \to \Sigma$ be an algebraic cellular automaton. 
Assume that one of the conditions $(\mathrm{H1})$, $(\mathrm{H2})$, $(\mathrm{H3})$ is satisfied. 
Then the following are equivalent: 
\begin{enumerate}[\rm (a)]
\item
$\tau$ is nilpotent;
\item 
$\tau$ is pointwise nilpotent; 
\item 
the limit set $\Omega(\tau)$ is finite. 
\end{enumerate}
If $G$ is finitely generated, then the above conditions are equivalent
to
\begin{enumerate}[{\rm (d)}]
\item 
each $x \in \Omega(\tau)$ is periodic and the set $\{x(1_G) \colon x \in \Omega(\tau)\}$ of alphabet values of $\Omega(\tau)$ is finite. 
\end{enumerate}
\end{theorem} 
\par 
Note that for classical cellular automata the equivalence of (a) and (b) does not require neither the topological mixing nor the soficity
conditions on the subshift $\Sigma \subset A^G$  (this is a result going back to Kari, \cite{salo-nilpotent-2012}). 
We don't know whether or not, in our more general setting, the above-mentioned conditions can be dropped. 
For classical cellular automata, the equivalence (b) and (c) is given in Theorem \ref{t:finit-limit-set}. 
Note that, however, if the alphabet $A$ is infinite, for any group $G$ there exist non-nilpotent cellular automata whose
limit set is reduced to a single configuration and therefore is finite (cf.\ Proposition \ref{p:ca-non-nil-limit-pt}).

A linear version of the above theorem was given in \cite[Theorem 1.9 and Corollary 1.10]{JPAA}.
Our general strategy evolves around the analysis of the so called \emph{space-time inverse system} 
associated with a cellular automaton  (cf.~Section~\ref{s:space-time-system}). 
Such inverse systems and their variants as constructed in the proofs of the main theorems   
allow us to first conduct a local analysis   
of the dynamical system 
as in Theorem~\ref{t:sofic-pro-constructible} and Theorem~\ref{t:space-time-ls-alg}. 
We can then pass to the inverse limit, 
by means of the key technical algebro-geometric tools Lemma~\ref{l:pro-constructible} and Lemma~\ref{l:inverse-limit-closed-im},  
in order to obtain global properties 
such as a closed mapping property in Theorem~\ref{t:closed-image} and 
a characterization of algebraic subshifts of finite type in Theorem~\ref{t:noetheiran-alg-sft}.  
Variants of space-time inverse systems also allow us to reduce 
Theorem~\ref{t:char-nilpotent-finite-alg-ca} to the finite alphabet case studied in Theorem~\ref{t:finit-limit-set}. 
\par 
We remark that by a similar strategy, it is shown in \cite{phung-group-2020} that for a polycyclic-by-finite group $G$ 
and an algebraic group $V$ over an algebraically closed field $K$, 
all algebraic group sofic subshifts of $A^G$, where $A=V(K)$, are in fact algebraic group subshifts of finite type. 
See also \cite{phung-dcds} and \cite{phung-post-surjective} for some further applications of our techniques.
\par 
Most of our results for arbitrary groups are inferred from the results for finitely generated groups 
by the restriction technique applied to cellular automata over \emph{subshifts of sub-finite-type} 
(cf.~Section~\ref{s:sofic-shift-def}, Section~\ref{s:restriction-ls}, Section~\ref{s:restriction-cip}). 
\par
A detailed analysis is given in Example~\ref{ex:limit-set-not-closed} to provide a nontrivial counter-example 
to Theorem~\ref{t:limit-set-not-empty} and Theorem~\ref{t:closed-image}. 
Some generalizations of our results are given in Section~\ref{s:extension}.  
In the Appendix, we study pointwise nilpotency over infinite groups and arbitrary alphabets (cf.~Proposition~\ref{p:carct-nilp-ca}).  
 
% SECTION 2
\section{Preliminaries}

\subsection{Notation}
We use the symbols  $\Z$ for the integers, $\N$ for the non-negative integers, $\R$ for the reals, and $\C$ for the complex numbers.
\par
We write  $A^B$ for the set consisting of all maps from a set $B$ into a set $A$. 
Let $C \subset B$. If $x \in A^B$,  
we denote by $x\vert_C$ the restriction of $x$ to $C$, that is, the map $x\vert_C \colon C \to A$ given by $x\vert_C(c) = x(c)$ for all $c \in C$.
If $X \subset A^B$, we denote $X_C \coloneqq \{ x\vert_C \colon x \in X \} \subset A^C$. 
Let $E,F$ be subsets of a group $G$.
We write $E F \coloneqq \{g h : g \in E, h \in F\}$
and  define inductively $E^n$ for all $n \in \N$ by setting $E^0 \coloneqq  \{1_G\}$ and $E^{n + 1} \coloneqq  E^n E$. 
\par
Let $A$ be a set and let $E$ be a subset of a group $G$. 
Given  $x \in A^E$, we define $gx \in A^{gE}$ by $(gx)(h) \coloneqq x(g^{-1}h)$ for all $h \in g E$. 

\subsection{Algebraic varieties} 
Let $V$ be an algebraic variety over a field $K$, that is, a reduced $K$-scheme of finite type. 
We equip $V$ with its Zariski topology. 
Every subset $Z \subset V$ is equipped with the induced topology 
and we denote by $Z(K)$ the subset of $K$-points of $V$ lying in $Z$. 
Subvarieties of $V$ mean closed subsets with the reduced induced scheme structure.  
\begin{remark} 
\label{r:h1-h2} 
Every subvariety of a complete  (i.e. proper) algebraic variety is also complete.  
Images of morphisms of complete algebraic varieties are complete subvarieties  
(cf.~\cite[Section~3.3.2]{liu-book}). 
Likewise, kernels and images of homomorphisms of algebraic groups are also algebraic subgroups and are thus Zariski  closed  
(cf.~\cite[Proposition~1.41,~Theorem~5.80,~Theorem~5.81]{milne-group-book}).  
\end{remark}

Suppose now that the base field $K$ is algebraically closed. 
Then we can identify 
the set of $K$-points $A=V(K)$ of a $K$-algebraic variety $V$ 
with the set of closed points of $V$ (cf.~\cite[Proposition~6.4.2]{ega-1}). 
By a common abuse, we regard $A$ as an algebraic variety. 
Similarly, induced maps on closed points by morphisms of $K$-algebraic varieties 
are also called algebraic morphisms. 

\subsection{Chain-recurrent points}

\begin{proposition}
\label{p:nw-included-in-crec}
Let $X$ be a uniform space and let $f \colon X \to X$ be a continuous map.
Then $\NW(f) \subset \Crec(f)$. 
\end{proposition}

\begin{proof}
(cf.~\cite[Proposition~1.7]{shub-global-stability} in the metrizable case) 
Let $x \in \NW(f)$ and let $E$ be an entourage of $X$. 
Choose a symmetric entourage $S$ of $X$ such that $S\circ S\subset E$. 
By the continuity of $f$ at $x$, 
there exists a symmetric entourage $T$ of $X$ with $T \subset S$  
such that $(f(x),f(z)) \in S$ whenever $(x,z) \in T$. 
The set $U \subset X$, consisting of all $z \in X$ such that $(x,z) \in T$, is a neighborhood of $x$. 
Since $x$ is non-wandering,
there exist an integer $n \geq 1$ and a point $y \in U$ such that $f^n(y) \in U$.
Let us show that there is a sequence of points $x_0,x_1,\dots,x_n \in X$ such that
$x= x_0 = x_n$ and  $(f(x_i),x_{i + 1}) \in E$ for all $0 \leq i \leq n - 1$.
First observe  that since $y \in U$ we have $(x,y) \in T$ and therefore $(f(x), f(y)) \in S$.
If $n=1$, we can take $x_0 = x_1 = x$.
Indeed, we then have $f(y) = f^n(y) \in U$ and hence $(f(y),x) \in T \subset S$.
Therefore $(f(x_0),x_1) = (f(x),x) \in  S \circ S \subset E$.
If $n \geq 2$, 
we can take  the  points $x_0,x_1,\dots,x_n$
defined by $x = x_0 = x_n$ and $x_i  = f^i(y)$ for all $1 \leq i \leq n -1$.
Indeed, we then have $(f(x_0),x_1) = (f(x),f(y)) \in S \subset S \circ S \subset E$.
On the other hand,  we have
$(f(x_i),x_{i+1}) = (f^{i + 1}(y),f^{i + 1}(y)) \in E$ for all $1 \leq i \leq n -2$.
Finally, as $f^n(y) \in U$, we have  $(f(x_{n - 1}),x_n) = (f^n(y),x) \in T \subset S \subset S \circ S \subset E$.   
This shows that $x \in \Crec(f)$.  
\end{proof}

\begin{proposition}
\label{p:chain-recurrent}
Let $X$ be a Hausdorff uniform space and let $f \colon X \to X$ be a uniformly continuous map.
Suppose that $f^n(X)$ is closed in $X$ for all $n \in \N$. 
Then $\Crec(f) \subset \Omega(f)$.
\end{proposition}

\begin{proof}
Denote by $\EE$ the set of entourages of $X$.
Let $x \in \Crec(f)$. 
Given  $E \in \EE$, we  define $\nu(E)  \in \N \setminus \{0\}$ to be the least $n \in \N$ such that there exists
a sequence of points $x_0,x_1,\dots,x_n \in X$ satisfying that $x = x_0 = x_n$ and $(f(x_i), x_{i + 1}) \in E$ for all $0 \leq i \leq n - 1$.
Note that the map $\nu \colon \EE \to \N \setminus \{0\}$  is decreasing in the sense that 
if  $E,E' \in \EE$ and $E \subset E'$,
then $\nu(E') \leq \nu(E)$.
We distinguish two cases according to whether 
the map  $\nu$ is bounded or not.
\par
In the first case, let $k \coloneqq \max \nu$.
Take $E_0 \in \EE$ such that $\nu(E_0) = k$.
Let $E \in \EE$.  
Choose  a symmetric entourage  $S \in \EE$ such that
\[
\underbrace{S \circ S \circ \cdots \circ S}_{k \mbox{ \tiny times}} \subset E.
\]
Since $f$ is uniformly continuous, so are $f^{2}, \dots, f^{k}$.
Thus, we can find a symmetric entourage $T \subset E_0$  such that $(f^p(y),f^p(z)) \in S$ whenever $(y,z) \in T$ and $0 \leq p \leq k$.
By the maximality of $k$ and the fact that $T \subset E_0$, we have $\nu(T) = k$. 
Therefore, we can find a sequence of  points  $x_0, x_1, \ldots, x_k \in X$ such that 
$x =x_0=x_k$ and $(f(x_i), x_{i+1}) \in T$ for all
$0 \leq i \leq k - 1$.
Looking at the sequence of points
$f^k(x) = f^k(x_0), f^{k - 1}(x_{1}), f^{k - 2}(x_{2}), \ldots, f^{1}(x_{k - 1}), x_k = x$ 
and using the fact that, for all $0 \leq i \leq k -1$,
\[
(f^{k - i}(x_{i}), f^{k - i - 1}(x_{i + 1})) = (f^{k - i - 1}(f(x_{i })), f^{k - i - 1}(x_{i + 1}))) \in S
\]
since $(f(x_{i }), x_{i + 1}) \in T$, 
We see that
\[
( f^k(x),x) \in \underbrace{S \circ S \circ \cdots \circ S}_{k \mbox{ \tiny times}} \subset E.
\]
As the entourage $E \in \EE$ was arbitrary and $X$ is Hausdorff, it follows that $x = f^k(x)$.
Hence,   the point $x$ is periodic and therefore belongs to $\Omega(f)$.
\par
Consider now the second case, where $\nu$ is unbounded. 
Let $m \geq 1$ be an integer. 
We will   show that $x \in f^m(X)$.
Take $E_0 \in \EE$  so that $\nu(E_0) \geq m$.
Let $E \in \EE$.
Choose a symmetric entourage $S \in \EE$ such~that
\[
\underbrace{S \circ S \circ \cdots \circ S}_{m \mbox{ \tiny times}} \subset E.
\]
As in the first case, we can find a symmetric entourage $T \in \EE$  such that $T \subset E_0$ and 
$(f^p(y),f^p(z)) \in S$
whenever $(y,z) \in T$ and $0 \leq p \leq m$.
Observe that $n \coloneqq \nu(T) \geq \nu(E_0) \geq m$ since $T \subset E_0$.
By definition of $\nu$,
we can find a sequence of points 
$x_0, x_1, \ldots, x_{n} \in X$ such that $x =x_0=x_n$ and 
$(f(x_i), x_{i+1}) \in T$ for all $0 \leq i \leq n -1$.
Looking now at the sequence of points
$f^m(x_{n - m}),f^{m - 1}(x_{n - m + 1}), \ldots, f(x_{n - 1}), x_n = x$, 
and using the fact  that, for all $0 \leq i \leq m -1$, we have
\begin{align*}
& (f^{m - i}(x_{n - m + i}), f^{m - i - 1}(x_{n - m + i + 1})) \\
& = (f^{m - i - 1}(f(x_{n - m + i})), f^{m - i - 1}(x_{n - m + i + 1}))  \in S
\end{align*}
since $(f(x_{n - m + i})),x_{n - m + i + 1}) \in T$,
we see  that
\[
(f^m(x_{n - m}), x)   \in \underbrace{S \circ S \circ \cdots \circ S}_{m \mbox{ \tiny times}} \subset E.
\]
As the entourage $E \in \EE$ was arbitrary, 
it follows  that $x$ belongs to the closure of $f^m(X)$.
Since  $f^m(X)$ is closed in $X$ by our hypothesis,
we conclude that  $x \in f^m(X)$ for every $m \geq 1$. 
This shows that $x \in \Omega(f)$.
\end{proof}

Using the fact that the topology of any compact Hausdorff space is induced by a  unique uniform structure,
an immediate consequence of Proposition~\ref{p:chain-recurrent} is the following well known result (see e.g.~\cite[Ch.~6]{osipenko-2007}).

\begin{corollary}
\label{c:chain-recurrent}
Let $X$ be a compact Hausdorff  space and let $f \colon X \to X$ be a  continuous map.
Then $\Crec(f) \subset \Omega(f)$.
\qed
\end{corollary}

\subsection{Subshifts of sub-finite-type}  
\label{s:sofic-shift-def}
Let $G$ be a group and let $A$ be a set. 
A subshift $\Sigma \subset A^G$ is called a \emph{subshift of sub-finite-type} if it is a factor of a subshift of finite type 
(cf.~\eqref{e:sft}), namely, there exist a set $B$, a cellular automaton $\tau' \colon B^G \to A^G$ and  
a subshift of finite type $\Sigma' \subset B^G$ such that $\Sigma= \tau'(\Sigma')$. 
Note that we do not require  $\Sigma$ to be closed in $A^G$. 
In the sequel, every finite subset $D$ of $G$ containing a defining memory set of $\Sigma'$ as well as a memory set of $\tau'$ 
 will be called a \emph{memory set} of the subshift of sub-finite-type $\Sigma$. 
 The existence of such a memory set will be necessary for the restriction technique 
 (cf.~Section~\ref{s:restriction-ls} and Section~\ref{s:restriction-cip}) when the group 
 $G$ is not finitely generated. 
 \begin{example} 
If $G$ is a group,  $V$ is an algebraic variety over a field $K$, and $A \coloneqq V(K)$,
then it immediately follows from Definition~\ref{d:alg-sofic-shift} in the Introduction that
every algebraic sofic subshift $\Sigma \subset A^G$ is a subshift of sub-finite-type of $A^G$. 
\end{example}
In the rest of the paper, 
a \emph{memory set} of an algebraic sofic subshift $\Sigma$ will mean any memory set 
of $\Sigma$ regarded as a subshift of sub-finite-type. 
\begin{example} 
Let $A$ be a set and let $\Gamma$ be an $A$-labelled directed graph.
This means that $\Gamma$ is a quintuple
$\Gamma  = (V,E,\alpha,\omega,\lambda)$, where $V,E$ are sets, and $\alpha,\omega \colon E \to V$, $\lambda \colon E \to A$  are maps.
The elements of $V$ are called the \emph{vertices} of $\Gamma$,
those of $E$ are called its \emph{edges},
and, for every edge $e \in E$, the vertex $\alpha(e)$ (resp.~$\omega(e))$ 
is called the \emph{initial} (resp.~\emph{terminal}) vertex of $e$
while $\lambda(e)$ is called its \emph{label}.
The \emph{label} of a configuration $x \in E^{\Z}$   is the configuration $\Lambda(x)  \in A^{\Z}$ 
defined by $\Lambda(x)(n) = \lambda(x(n))$ for all $n \in \Z$. 
Observe that $\Lambda \colon E^{\Z} \to A^{\Z}$ is a cellular automaton 
admitting $M \coloneqq \{0\} \subset \Z$ as a memory set 
 and $\lambda \colon E^M = E \to A$ as the associated local defining map. 
 An element  $x \in E^{\Z}$ is called a \emph{path} of $\Gamma$
if it satisfies $\omega(x(n)) = \alpha(x(n+1))$ for all $n \in \Z$. 
Clearly, the subset $\Sigma' \subset E^{\Z}$ consisting of all paths of $\Gamma$
is the subshift of finite type $\Sigma(D,P)$ of $E^{\Z}$ where $D\coloneqq \{0,1\} \subset \Z$
and $P \coloneqq \{p \in E^D: \omega(p(0)) = \alpha(p(1))\}$.
One says that $\Sigma'$ is the \emph{Markov shift} associated with the unlabelled graph $(V,E,\alpha,\omega)$
(cf.~\cite[Ch.~7]{kitchens-book}).
We deduce that  $\Sigma \coloneqq \Lambda(\Sigma')$ is a subshift of sub-finite-type of $A^\Z$. 
Conversely, it can be shown that every subshift of sub-finite-type of $A^{\Z}$ can be obtained, 
up to topological conjugacy, as the set of labels of the paths of a suitably chosen $A$-labelled graph.
The proof of this last result is, mutatis mutandis,
the one used in the classical setting
for showing that every sofic finite alphabet subshift over $\Z$
can be presented by a finite labelled graph (see e.g.~\cite[Theorem~3.2.1]{lind-marcus}). 
\end{example}

The following result says that the notion of subshifts of sub-finite-type 
is only interesting when $G$ is not finitely generated. 

\begin{proposition} 
\label{p:shift-general-alphabet} 
Let $G$ be a finitely generated group and let $A$ be a set. 
Then every subshift $\Sigma \subset A^G$ is a subshift of sub-finite-type.   
\end{proposition}

\begin{proof}
Let $D$ be a finite generating subset of $G$ such that $1_G \in D$ and $D=D^{-1}$. 
Let $\Sigma \subset A^G$ be a subshift. 
Let $B \coloneqq \Sigma$ and define 
$P \coloneqq \{ y\in B^D \colon y(g) = g^{-1}(y(1_G)) \mbox{ for all } g \in D \}$. 
Consider the subshift of finite type $\Sigma' \coloneqq \Sigma(D,P)$ of $B^G$. 
Since $D=D^{-1}$ generates $G$ and contains $1_G$,  
the map $\mathfrak{X} \mapsto \mathfrak{X}(1_G)$ is a bijection from  $\Sigma'$ onto  $\Sigma$. 
Indeed, for every $g \in G$ and $\mathfrak{X} \in \Sigma'$, by writing $g= s_1 \dots s_n$ for 
some $s_1, \dots, s_n \in D$, we find that:   
\[
\mathfrak{X}(g) = \mathfrak{X}(s_1 \dots s_n) 
= s_1^{-1} \mathfrak{X}(s_2 \dots s_n) = \dots 
= s_n^{-1} \dots s_1^{-1} (\mathfrak{X}(1_G)) = g^{-1} \mathfrak{X}(1_G). 
\]
Let $\tau \colon B^G \to A^G$ be the cellular automaton
with memory set $\{1_G\}$ and associated local defining map $\mu \colon B \to A$ given by $x \mapsto x(1_G)$. 
In other words, $\tau(\mathfrak{X})(g) = (\mathfrak{X}(g))(1_G)$ for every $\mathfrak{X} \in B^G$
and $g \in G$. Hence, for every $\mathfrak{X} \in \Sigma'$, we have 
$\tau(\mathfrak{X}) = \mathfrak{X}(1_G)$ since for all $g \in G$,  
\[
\tau(\mathfrak{X}) (g) =(\mathfrak{X}(g))(1_G) = (g^{-1} (\mathfrak{X}(1_G)) )(1_G) 
= (\mathfrak{X}(1_G))(g). 
\] 
As $\mathfrak{X}(1_G) \in B= \Sigma$ is arbitrary, we conclude that $\Sigma=\tau(\Sigma')$ 
is a subshift of sub-finite-type.  
\end{proof}

\subsection{Restriction of cellular automata and of subshifts of sub-finite-type}
\label{s:restriction-ls} 
Let $G$ be a group and let $A$ be a set. 
Let $\Sigma \subset A^G$ be a subshift of sub-finite-type. 
Hence, there exist a set $B$, a cellular automaton $\tau' \colon B^G \to A^G$ and  
a subshift of finite type $\Sigma' \subset B^G$ such that $\Sigma= \tau'(\Sigma')$. 
Let $D \subset G$ be a finite subset 
such that $D$ is a defining memory set of $\Sigma'$ as well as a memory set of $\tau'$. 
Let  $H\subset G$ be  a subgroup of $G$ containing $D$. 
Denote by $G/H \coloneqq  \{gH \colon g \in G\}$ the
set of all right cosets of $H$ in $G$.
As the right cosets of $H$ in $G$ form a partition of $G$,
we have natural factorizations 
\[
A^G = \prod_{c \in G/H} A^c, \quad B^G = \prod_{c \in G/H} B^c
\]
in which each $x \in A^G$ (resp.~$x \in B^G$) 
is identified with  $ (x\vert_c)_{c \in G/H} \in \prod_{c \in G/H} A^c$ 
(resp.~$(x\vert_c)_{c \in G/H} \in \prod_{c \in G/H} B^c$). 
Since  $g D \subset gH$ for every $g \in G$,
the above factorization of  $B^G$ induces a factorization 
\[
\Sigma' = \prod_{c \in G/H} \Sigma'_c,
\]
where $\Sigma'_c = \{x\vert_c \colon x \in \Sigma'\}$ for all $c \in G/H$. 
Likewise, for each $c \in G/H$, let $\Sigma_c = \{x\vert_c \colon  x \in \Sigma\}$. 

\begin{lemma}
\label{l:factorization-sigma} 
The factorization $A^G = \prod_{c \in G/H} A^c$ induces
a factorization 
\[
\Sigma = \prod_{c \in G/H} \Sigma_c. 
\]
\end{lemma} 
\begin{proof}
Since $H$ contains a memory set of $\tau'$, 
we have  $\tau' = \prod_{c \in G/H} \tau'_c$,
where $\tau'_c \colon B^c \to A^c$ is given by 
$\tau'_c(y) \coloneqq \tau'(x)\vert_c$ for all $y \in B^c$,
where $x \in B^G$ is any configuration extending $y$. 
We deduce that 
$\Sigma_c = (\tau'( \Sigma'))_c = \tau'_c(\Sigma'_c)$ for every $c \in G/H$. 
 Hence, 
\begin{align*}
\Sigma = \tau'(\Sigma') = \tau' \left( \prod_{c \in G/H} \Sigma'_c \right) 
= \prod_{c \in G/H} \tau'_c(\Sigma'_c) = \prod_{c \in G/H} \Sigma_c. 
\end{align*} 
\end{proof}
Let $T \subset G$ be a complete set of representatives for the right
cosets of $H$ in $G$ such that $1_G \in T$.
Then, for each $c \in G/H$, we have a uniform homeomorphism 
$\phi_c \colon \Sigma_c \to \Sigma_H$ given by
$\phi_c(y)(h) = y(gh)$ for all $y \in \Sigma_c$, where $g \in T$
represents $c$.
In particular, $\Sigma \neq \varnothing$ if and only if $\Sigma_H \neq
\varnothing$. 
\par 
Now suppose in addition that $\tau \colon \Sigma \to \Sigma$ is a cellular automaton
which admits a memory set contained in $H$.
Then we have  $\tau = \prod_{c \in G/H} \tau_c$,
where $\tau_c \colon \Sigma_c \to \Sigma_c$ is defined by setting
$\tau_c(y) \coloneqq \tau(x)\vert_c$ for all $y \in \Sigma_c$,
where $x \in \Sigma$ is any configuration extending $y$. 
Note that for each $c \in G/ H$,
the maps $\tau_c$ and $\tau_H$ are conjugate by $\phi_c$, i.e.,  we have 
$\tau_c = \phi_c^{-1} \circ \tau_H \circ \phi_c$.
This allows us to identify the action of  $\tau_c$ on $\Sigma_c$
with that of  the restriction cellular
automaton $\tau_H$  on $\Sigma_H$.

\begin{lemma}
\label{l:restriction-ls}
The following hold:
\begin{enumerate} [\rm (i)]
\item
$\Omega(\tau) = \Omega(\tau_H)^{G/H}$;
\item
$\tau$ is nilpotent if and only if $\tau_H$ is nilpotent.
\end{enumerate}
\end{lemma}

\begin{proof}
Observe that the map $x \mapsto (\phi_c(x\vert_c))_{c \in G/H}$ yields a bijection
$\Omega(\tau) \to \prod_{c \in G/H} \Omega(\tau_H) =
\Omega(\tau_H)^{G/H}$, and this proves (i).  The point (ii) is clear by
the above discussion.
\end{proof}

\subsection{Restriction and the closed image property} 
\label{s:restriction-cip} 
Let $G$ be a group and let $A, B$ be sets. 
Let $\Sigma \subset A^G$ be a subshift of sub-finite-type  
and let now $\tau \colon A^G \to B^G$ be a cellular automaton
whose source and domain are the full shifts $A^G$ and $B^G$ respectively.  
Let $H\subset G$ be a subgroup of $G$ containing a memory set of $\Sigma$
and a memory set of $\tau$.  
As in Section~\ref{s:restriction-ls}, we have the factorizations
$\Sigma = \prod_{c \in G/H} \Sigma_c$ 
(cf.~Lemma~\ref{l:factorization-sigma})
and  
$\tau = \prod_{c \in G/H} \tau_c$,
with $\tau_c \colon A^c \to B^c$ defined by 
$\tau_c(y) \coloneqq \tau(x)\vert_c$ for all $y \in A^c$,
where $x \in A^G$ is any configuration extending $y$. 

\begin{lemma}
\label{l:restriction-cip}
The set $\tau(\Sigma)$ is closed in $B^G$
if and only if $\tau_H(\Sigma_H)$ is closed in $B^H$. 
\end{lemma}
\begin{proof} 
We have  
$\tau(\Sigma)= \prod_{c \in G/H} \tau_c(\Sigma_c)$. 
It is immediate that
$\tau_H(\Sigma_H)$ is closed in $B^H$ if $\tau(\Sigma)$ is closed in $B^G$. 
For the converse implication, we have for every $c \in G/H$ 
a uniform homeomorphism 
$\psi_c \colon B^c \to B^H$ by fixing a complete set containing $1_G$ 
of representatives for the right cosets of $H$ in $G$  (cf.~Section~\ref{s:restriction-ls}). 
Thus, if $\tau_H(\Sigma_H)$ is closed then so is 
$\tau_c(\Sigma_c) = \psi_c^{-1}(\tau_H(\Sigma_H))$. 
Consequently, $\tau(\Sigma)$ is closed in $B^G$ whenever $\tau_H(\Sigma_H)$ is closed in $B^H$ 
since the product of closed subspaces is closed in the product topology. 
\end{proof}

% SECTION 3
\section{Inverse limits of countably pro-constructible sets} 
\label{s:countably-pro-constructible}
Let $I$ be a directed set, i.e., a partially ordered set in which every pair of elements admits an upper bound.
An \emph{inverse system} of sets  \emph{indexed} by  $I$ consists of the following data:
(1) a set $Z_i$ for each $i \in I$;
(2)  a
\emph{transition map}
$\varphi_{ij} \colon Z_j \to Z_i$
for all $i,j \in I$ such that $i \prec j$.
Furthermore, the transition maps must satisfy  the following conditions:
\begin{align*}
 \varphi_{ii} &= \Id_{Z_i} \text{ (the identity map on $Z_i$) for all } i \in  I, \\[4pt]
 \varphi_{ij} \circ \varphi_{jk}  &= \varphi_{ik}  \text{ for all $i,j,k \in I$ such that } i \prec j \prec k. %\\
\end{align*}
One then speaks of the inverse system $(Z_i,\varphi_{ij})$, or simply  $(Z_i)$ if
the index set and the transition maps  are clear from the context.
 \par
The \emph{inverse limit} of an inverse system $(Z_i,\varphi_{i j})$ is the subset
\[
\varprojlim_{i \in I} (Z_i,\varphi_{i j}) =  \varprojlim_{i \in I} Z_i \subset \prod_{i \in I} Z_i
\]
consisting of all  $(z_i)_{i \in I}$ such that $\varphi_{i j}(z_j)= z_i$ for all $i \prec j$. 
\par
 A subset of a topological space $X$ is said to be \emph{locally closed} if it is the intersection of a closed subset and an open subset of $X$.
It is said to be \emph{constructible} if it is a finite union of locally closed subsets of $X$.
It is said  to be \emph{proconstructible} if it is the intersection of a family of constructible subsets~\cite[D\' efinition~I.9.4]{grothendieck-20-1964}. 
We shall say that a subset of $X$ is \emph{countably-proconstructible}
if it is the intersection of a countable family of constructible subsets. 
It is clear that every countably-proconstructible subset can be written as the intersection of a decreasing sequence of constructible subsets.
\par  
The following lemma is analogous to \cite[Lemma 4.1]{phung-2018}. 

\begin{lemma} 
\label{l:pro-constructible}
Let $K$ be an uncountable algebraically closed field and 
let $f \colon X \to Y$ be an algebraic morphism of algebraic varieties over $K$. 
If $(C_k)_{k \in \N}$ is a decreasing sequence of constructible subsets of $X$, 
then 
\[
f \left( \bigcap_{k \in \N} C_k(K) \right) = \bigcap_{k \in \N} f(C_k(K)) =  \bigcap_{k \in \N} f(C_k)(K). 
\]
\end{lemma} 

\begin{proof}
Since for each $k \in \N$, 
we have $f(C_k(K))= f(C_k)(K)$ (cf.~for example \cite[Lemma~A.22.(v)]{ccp-2019}), the second equality is verified.
For the first equality, we have trivially 
$f \left( \bigcap_{k \in \N} C_k(K) \right) \subset \bigcap_{k \in \N} f(C_k(K))$. 
Conversely, assume that $y \in \bigcap_{k \in \N} f(C_k(K))$. 
For each $k \in \N$, set 
\[ 
F_k \coloneqq   f^{-1}(y) \cap C_k (K)   \subset X (K). 
\] 
Note that $F_k$ is the set of closed points of a constructible subset of $X$. 
Remark also  that, for every $k \in \N$, 
we have $F_{k+1} \subset F_k$ and $F_k \not= \varnothing$. 
Hence, by \cite[Lemma B.3]{ccp-2019}, there exists $x \in \bigcap_{k \in \N} F_k$. 
Clearly, $f(x)=y$ and $x \in \bigcap_{k \in \N} C_k(K)$. Therefore,  
$ \bigcap_{k \in \N} f(C_k(K)) \subset f \left( \bigcap_{k \in \N} C_k(K) \right)$ and the proof is completed. 
\end{proof}

In the case (H1), we shall make use of the following generalization of \cite[Lemma B.2]{ccp-2019} 
to countable inverse systems of countably-proconstructible subsets.

\begin{lemma} 
\label{l:inverse-limit-seq-const}
Let $K$ be an uncountable algebraically closed field. 
Let $(X_i, f_{ij})$ be an  inverse system indexed by a countable directed set $I$, where   
each $X_i$ is a $K$-algebraic variety and each transition map  
$f_{i j} \colon X_j \to X_i$ is an algebraic morphism.
Suppose given,
for each $i \in I$, a nonempty countably-proconstructible subset $C_i \subset X_i$. 
Let $Z_i= C_i(K)$ and assume that $f_{ij}(Z_j)\subset Z_i$ for all  $i\prec j$ in $I$. 
Then the inverse system $(Z_i,\varphi_{i j})_I$,
where $\varphi_{i j} \colon Z_j \to Z_i$ is the restriction of $f_{i j}$ to $Z_j$, 
verifies $\varprojlim_{i \in I}  Z_i \neq \varnothing$.
\end{lemma} 

\begin{proof} 
Since $I$ is a countable directed set, we can find a totally ordered cofinal subset  $\{i_n : n \in \N\} \subset I$. 
 As $\varprojlim_{n \in \N} Z_{i_n} = \varprojlim_{i \in I} Z_i$, 
we can suppose, without any loss of generality, that $I= \N$. 
\par 
For each $i \in \N$, we can find a decreasing sequence of constructible subsets $(C_{ik})_{k \in \N}$ of $X_i$  
such that $C_i= \bigcap_{k \in \N} C_{ik}$. 
For $k \in \N$, let $Z_{ik}= C_{ik}(K)$. By Lemma \ref{l:pro-constructible}, 
we have for every $i \leq j$: 
\begin{equation}
\label{e:pro-constructible-1}
Z_i= \bigcap_{k = 0}^{\infty} Z_{ik} \neq \varnothing, \quad f_{ij}(Z_j) = \bigcap_{k=0}^{\infty} f_{ij} (Z_{jk}). 
\end{equation}
Consider the universal inverse system $(Z'_i, \varphi'_{ij})_{i,j \in \N}$ of the system 
$(Z_i, \varphi_{ij})_{i,j \in \N}$, i.e.,  for every $i \in \N$, let
\[ 
Z'_i \coloneqq \bigcap_{j = i}^{\infty} f_{ij}(Z_{j}) =\bigcap_{j = i}^{\infty} \varphi_{ij}(Z_{j}) 
\] 
and let the maps $\varphi'_{ij} \colon Z'_{j} \to Z'_i$ be the restrictions of $\varphi_{ij} \colon Z_j \to Z_i$. 
\par 
Remark that $\varprojlim_{i \in \N} Z'_i= \varprojlim_{i \in \N} Z_i$. 
Hence, it suffices to check that the sets $Z'_i$ are nonempty and the transition maps $\varphi'_{ij}$ are surjective for all $i \leq j$.  
By \eqref{e:pro-constructible-1}, Chevalley's theorem (see for example~\cite[Theorem~7.4.2]{vakil}, \cite[Th\' eor\` eme~I.8.4]{grothendieck-20-1964}) 
implies that each $Z'_i$ is a countable intersection of constructible sets: 
\[ 
Z'_i = \bigcap_{j=i}^\infty f_{ij}(Z_{j})  
= \bigcap_{j = i}^\infty \bigcap_{k=0}^\infty   f_{ij} (Z_{jk}). 
\] 
For each $n \geq i$, consider the diagonal set 
\[
Y_n \coloneqq  \bigcap_{j=i}^n \bigcap_{k = 0}^n f_{ij}(Z_{jk}) \subset X_i(K). 
\] 
By Chevalley's theorem, $Y_n$ is a constructible subset of $X_i(K)$. 
For every $n \geq i$, we have $Y_{n+1} \subset Y_n$ and since $Z_n \ne \varnothing$,  
\begin{equation}
\label{e:non-empty-1}
Y_n \supset  \bigcap_{j=i}^n \bigcap_{k = 0}^\infty f_{ij}(Z_{jk})
= \bigcap_{j=i}^n f_{ij} (Z_j) 
\supset f_{in}(Z_n) 
= \varphi_{in}(Z_n)
\neq \varnothing. 
\end{equation}
As $Z'_i = \bigcap_{n=i}^\infty Y_n$,   \cite[Lemma~B.2]{ccp-2019} implies that 
$Z'_i \neq \varnothing$ for $i \in \N$.  
Now let $k, i \in \N$ with $k \leq i$ and let $z \in Z'_k$.  
For each $n \geq i$, by definition  of $Z'_k$,  
there exists $y \in Z_n$ such that $\varphi_{k n}(y)=z$ and thus  
\begin{equation}
\label{e:non-empty-2}
\varphi_{i n}(y)\in \varphi_{k i}^{-1} (z) \cap \varphi_{i n} (Z_n) \neq \varnothing.
\end{equation}
By \eqref{e:non-empty-1}, \eqref{e:non-empty-2}  
 and for   $n \geq i$, the constructible subset 
\begin{equation}
\label{e:non-empty-3}
 T_n \coloneqq \varphi_{k i}^{-1} (z) \cap Y_n 
 \supset  \varphi_{k i}^{-1} (z) \cap \varphi_{in}(Z_n) 
\end{equation} 
is nonempty and $T_{n+1} \subset T_n$ as $Y_{n+1} \subset Y_n$. Finally, we find that: 
\begin{align*} 
(\varphi_{k i}')^{-1} (z) 
& =  \varphi_{k i}^{-1} (z)  \cap Z'_i 
= \bigcap_{n=i}^\infty  \varphi_{k i}^{-1} (z)  \cap Y_n = \bigcap_{n=i}^\infty T_n 
\end{align*} 
is nonempty by \cite[Lemma~B.2]{ccp-2019}. The proof is thus completed.  
\end{proof} 

We shall apply repeatedly the following result in the cases (H2)-(H3).   

\begin{lemma}
\label{l:inverse-limit-closed-im}
Let $K$ be an  algebraically closed field.
Let $(X_i, f_{ij})$ be an  inverse system indexed by a countable index set $I$, 
where    each $X_i$ is a nonempty  $K$-algebraic variety and each transition map  
$f_{i j} \colon X_j \to X_i$ is an   algebraic morphism
such that $f_{ij}(X_j) \subset X_i$ is a closed subset for all $i \prec j$. 
Then $\varprojlim_{i \in I} X_i(K) \neq \varnothing$.
\end{lemma} 

\begin{proof}
The statement is proved in \cite[Proposition 4.2]{phung-2018}. 
\end{proof}

% SECTION 4
\section{Space-time inverse systems}  
\label{s:space-time-system} 

Let $G$ be a finitely generated group and let $A$ be a set. 
Let $\Sigma \subset A^G$ be a closed subshift  and assume that  $\tau \colon \Sigma \to \Sigma$ is a cellular automaton.
Let $\widetilde{\tau} \colon A^G \to A^G$
be a cellular automaton extending $\tau$. 
\par 
Let $M\subset G$ be a  memory set of $\widetilde{\tau}$.
Since every finite subset of $G$ containing a memory set of $\widetilde{\tau}$  is itself a memory set of $\widetilde{\tau}$,
we can choose $M$  such that 
$1_G \in M$,   $M = M^{-1}$, and $M$ generates $G$. 
Note that this implies in particular that the sequence $(M^n)_{n \in \N}$ is an \emph{exhaustion} of $G$, that is,
\begin{enumerate}[(Mem1)]
\item 
  $M^{n+1} \supset M^n$ for all $n \in \N$ and
  \item
   $\bigcup_{n \in \N} M^n = G$.
   \end{enumerate}
      \par
Equip $\N^2$ with the product ordering  $\prec$. 
Thus, given $i,j,k,l \in \N$, we have  $(i,j) \prec (k,l)$ if and only if $i \leq k$ and $j \leq l$.
\par
We construct  an inverse system $(\Sigma_{i j})_{i,j \in \N}$ indexed by  the directed set $(\N^2, \prec)$ 
in the following way.
\par
Firstly, given $i,j \in \N$, we define  $\Sigma_{i j}$
as being the set consisting of the restrictions to $M^{i + j}$ of all the configurations that belong to  $\Sigma$, that is, 
\[
\Sigma_{i j} \coloneqq \Sigma_{M^{i+j}} = \{ x\vert_{M^{i  + j}} : x \in \Sigma \} \subset A^{M^{i + j}}.
\]
To  define the transition maps
$\Sigma_{k l} \to \Sigma_{i j}$ ($(i,j) \prec (k,l)$) 
of the inverse system $(\Sigma_{ij})_{i,j \in \N}$,
it is clearly enough  to define, for all $i,j \in \N$,  the
 \emph{unit-horizontal} transition  map 
$p_{i j} \colon \Sigma_{i+1, j}  \to \Sigma_{i j}$,
 the \emph{unit-vertical} transition map $q_{i j} \colon \Sigma_{i,j+1} \to \Sigma_{i j}$,
 and verify that the diagram 
 \[
\begin{tikzcd}
%\label{d-abelian-family-diagram}
\Sigma_{i, j+1}  
  \arrow[d, swap, "q_{ij}"]  
&  \Sigma_{i+1, j+1}  \arrow[l, swap,  "p_{i, j+1}"]  \arrow[d, "q_{i+1,j}"]   \\ 
\Sigma_{ij}   & \Sigma_{i+1, j} \arrow[l, "p_{ij}"],   
\end{tikzcd}
\] 
is commutative, i.e.,
\begin{equation}
\label{e:p-q-commute}
q_{i j} \circ p_{i,j+1} = p_{i j} \circ q_{i+1,j} \quad \text{ for all } i,j \in \N. 
\end{equation}
We define $p_{i j}$ as being the map obtained by restriction to $M^{i + j} \subset M^{i + j + 1}$.
Thus for all $\sigma \in \Sigma_{i + 1,j}$, we have 
\begin{equation}
\label{e:def-p-ij}
p_{i j}(\sigma) = \sigma\vert_{M^{i + j}}. 
\end{equation} 
To define $q_{i j}$, we first observe that, 
given $x \in \Sigma$ and $g \in G$,
it follows from~\eqref{e;local-property} applied to $\widetilde{\tau}$ that
$\tau(x)(g)$ only depends on the restriction of $x$ to $gM$.
As $gM \subset M^{i + j + 1}$ for all $g \in M^{i + j}$,
we deduce from this observation that, given $\sigma \in \Sigma_{i,j + 1}$ and $x \in \Sigma$ extending $\sigma$,
the formula
\begin{equation}
\label{e:def-q-ij}
q_{i j}(\sigma) \coloneqq (\tau(x))\vert_{M^{i + j}} 
\end{equation}
yields a well-defined element $q_{i j}(\sigma) \in \Sigma_{i j}$
and hence a map $q_{i j} \colon \Sigma_{i,j+1} \to \Sigma_{i j}$.
\par
To check that~\eqref{e:p-q-commute} is satisfied,
let $\sigma \in \Sigma_{i+1,j+1}$ and choose a configuration $x \in \Sigma$ extending $\sigma$.
By applying~\eqref{e:def-p-ij},
we see that $p_{i,j+1}(\sigma) = x\vert_{M^{i + j + 1}}$.
Therefore, using~\eqref{e:def-q-ij}, we get 
\begin{equation}
\label{e:p-q-commute-1}
q_{i j} \circ p_{i,j+1} (\sigma) = q_{i j} ( p_{i,j+1}(\sigma)) = q_{i j} (x\vert_{M^{i + j + 1}}) = (\tau(x))\vert_{M^{i + j}}.
\end{equation}
On the other hand, by applying again~\eqref{e:def-q-ij}, we see that
$q_{i+1,j} (\sigma) = (\tau(x))\vert_{M^{i + j + 1}}$.
Therefore, using~\eqref{e:def-p-ij}, we get 
  \begin{equation}
\label{e:p-q-commute-2}
  p_{i j} \circ q_{i+1,j} (\sigma) = p_{i j} ( q_{i+1,j} (\sigma)) = p_{i j} ((\tau(x))\vert_{M^{i + j + 1}}) = (\tau(x))\vert_{M^{i + j }}.
  \end{equation}
We deduce from~\eqref{e:p-q-commute-1} and~\eqref{e:p-q-commute-2} that
  $q_{i j} \circ p_{i,j+1} (\sigma) = p_{i j} \circ q_{i+1,j} (\sigma) $ for all $\sigma \in \Sigma_{i + 1,j + 1}$. This shows~\eqref{e:p-q-commute}.

  \begin{definition}
\label{d:space-time-system}  
The inverse system $(\Sigma_{ij})_{i,j \in \N}$ is called the \emph{space-time inverse system} 
associated with the triple $(\Sigma,\tau,M)$. 
\end{definition} 

It might be useful to consider the inverse system $(\Sigma_{ij})_{i,j \in \N}$ as a refined diagram of the space-time evolution 
of the cellular automaton $\tau$ that in addition keeps track of the local dynamics. 
Comparing to the usual space-time diagram of a classical cellular automaton introduced in~\cite{wolfram-new-kind}, 
in~\cite{milnor-ca-1988} or recently 
in~\cite{cyr-kra-space-2019}, the main difference of our construction is the following. 
First, the horizontal direction indexed by $i \in \N$ in our space-time inverse system represents the extension 
of the ambient spaces of $1_G$ instead of the exact position in the universe $G$ as in the classical diagram. 
Second, the vertical direction indexed by $j \in \N$ represents the past instead of the future.  
More precisely, let us fix $i \in \N$ and consider the induced inverse subsystem $(\Sigma_{ij})_{j \in \N}$ lying above 
$\Sigma_{i0}= \Sigma_{M^i}$. Then each $(\Sigma_{ij})_{j \in \N}$ should be regarded as an approximation of 
the past light cone of the events happening in $M^i$, i.e., 
of confugurations $\sigma \in \Sigma_{M^i}$. 
 
\begin{remark}
\label{r:notation-full-shift}
Observe that the inverse system $(\Sigma_{i j})_{i,j \in \N}$
is a subsystem of the \emph{full} inverse system $(A^{M^{i + j}})_{i,j \in \N}$ associated with the triple
$(A^G,\widetilde{\tau},M)$.
In the sequel, 
we shall denote by
$\widetilde{p} _{ij}\colon A^{M^{i+j+1}} \to A^{M^{i+j}}$ (resp.~$\widetilde{q}_{ij} \colon A^{M^{i+j+1}} \to A^{M^{i+j}}$)
the unit horizontal (resp.~vertical) transition maps of the full inverse system    
$(A^{M^{i + j}})_{i,j \in \N}$.
\end{remark} 

\begin{remark}
\label{r:full-system-homomorphism}
For the hypotheses (H1), (H2) and (H3) in the Introduction, we have the following easy but useful remark.  
If $A=V(K)$ for some algebraic variety $V$ over an algebraically closed field $K$ and if 
$\widetilde{\tau} \colon A^G \to A^G$ is an algebraic (resp. algebraic group) cellular automaton, then 
the transition maps of the full inverse system $(A^{M^{i + j}})_{i,j \in \N}$ 
are algebraic morphisms (resp. homomorphisms of algebraic groups). 
\end{remark}

If we fix $j \in \N$ in our space-time  inverse system,
we get a \emph{horizontal} inverse system $(\Sigma_{i j})_{i \in \N}$ indexed by $\N$
whose transition maps are the restriction maps $p_{i j} \colon \Sigma_{i + 1,j} \to \Sigma_{i j}$, $i \in \N$.
It immediately follows from the closedness of $\Sigma$ in $A^G$ and properties (Mem1)-(Mem2) that the~limit
\begin{equation}
\label{e:row-limit}
\Sigma_j \coloneqq \varprojlim_{i \in \N} \Sigma_{ij} 
\end{equation}
can be identified with $\Sigma$ in a canonical way.
Moreover, the maps $q_{i j} \colon \Sigma_{i, j + 1} \to \Sigma_{i j}$
define an inverse system morphism from
the inverse system $(\Sigma_{i, j + 1})_{i \in \N}$ to the inverse system $(\Sigma_{i j})_{i \in \N}$.
This yields a limit map $\tau_j \colon \Sigma_{j + 1} \to \Sigma_j$.
Using the identifications $\Sigma_{j + 1} = \Sigma_j = \Sigma$, we have $\tau_j = \tau$ for all $j \in \N$.
We deduce that 
the limit
\begin{equation}
\label{e:backward-orbits}
\varprojlim_{i,j \in \N} \Sigma_{ij} = \varprojlim_{j \in \N} \Sigma_{j}
\end{equation} 
is the set of \emph{backward orbits} (or \emph{complete histories}~\cite{milnor-ca-1988}) of $\tau$, that is, 
the set consisting of all sequences $(x_j)_{j \in \N}$ such that $x_j \in \Sigma$ and $x_j = \tau(x_{j+1})$ for all $j \in \N$. 
Such a sequence satisfies $x_0 = \tau^n(x_n)$ for all $n \in \N$ and hence $x_0 \in \Omega(\tau)$.
Thus we obtain the following result.
\begin{lemma} 
\label{l:space-time-ls}
We have a canonical map $\Phi \colon \varprojlim_{i,j \in \N} \Sigma_{ij}   \to \Omega(\tau)$. 
In particular, we have that: 
\[
\varprojlim_{i,j \in \N} \Sigma_{ij}  \neq \varnothing \implies \Omega(\tau)  \neq \varnothing.
\]
\qed 
\end{lemma} 
We will see that the map $\Phi \colon \varprojlim_{i,j \in \N} \Sigma_{ij}   \to \Omega(\tau)$ 
is surjective in the algebraic setting (cf.~Theorem \ref{t:space-time-ls-alg}). 
Therefore, in this case, every limit configuration $x \in \Omega(\tau)$ admits a backward orbit 
and $\tau(\Omega(\tau)) = \Omega(\tau)$. 

% SECTION 5
\section{Approximation of subshifts of finite type}
\label{s:sft}
In this section, keeping all the notations and hypotheses introduced in the previous section,   
we assume in addition  that $\Sigma$ is a subshift of finite type.
We  fix a finite subset $D \subset G$ and a subset $P \subset A^D$ such that
$\Sigma = \Sigma(D, P)$ (cf. \eqref{e:sft}). We begin with a useful observation: 
\begin{lemma}
\label{l:base-change-sft}
For every finite subset $E \subset G$ such that $D \subset E$, we have 
$\Sigma = \Sigma(D, P) = \Sigma(E, \Sigma_E)$. 
\end{lemma}
\begin{proof} 
Let $x \in \Sigma$ and $g \in G$, then clearly $(g^{-1}x) \vert_E \in \Sigma_E$. 
Thus $\Sigma \subset \Sigma(E, \Sigma_E)$. Conversely, let 
$x \in \Sigma(E, \Sigma_E)$ and $g \in G$, then 
$(g^{-1}x)\vert_D =  ((g^{-1}x)\vert_E)\vert_D\in (\Sigma_E)_D \subset P$ since $D \subset E$. 
Therefore, $x \in \Sigma(D,P)=\Sigma$ and the conclusion follows. 
\end{proof}
 
For all $i, j \in \N$, we define $D_{ij} \coloneqq \{ g\in G \colon gD \subset M^{i+j} \}$ and 
\begin{equation}
\label{e:a-ij} 
A_{ij} \coloneqq \{x \in A^{M^{i+j}} \colon (g^{-1} x)\vert_D \in P \text{ for all } g\in D_{ij} \}. 
\end{equation}

Remark that  $\Sigma_{ij}  \subset A_{ij}$. 
Indeed, we have
\begin{align*}
\Sigma_{ij} & =  \{ x \in A^{M^{i+j}} \colon \exists y \in A^G, x = y \vert_{M^{i+j}}, (g^{-1}  y)\vert_{D} \in P \text{ for all } g \in G\}  \\ 
& \subset  \{ x \in A^{M^{i+j}} \colon \exists y \in A^G, x = y \vert_{M^{i+j}}, (g^{-1}  y)\vert_{D} \in P \text{ for all } g \in D_{ij}\} \\ 
&= \{ x \in A^{M^{i+j}} \colon (g^{-1} x)\vert_D \in P \text{ for all } g\in D_{ij} \} \\ 
& = A_{ij}. 
\end{align*} 
Remark also that for all $(i, j) \prec (k, l)$ in $\N^2$, 
we have $D_{ij} \subset D_{kl}$ because $M^{i+j} \subset M^{k+l}$ by (Mem1). 
\par
For $i,j,k \in \N$ such that $i \leq k$, consider the canonical projection 
\begin{equation} 
\widetilde{p}_{ijk} \colon A^{M^{k+j}} \to A^{M^{i+j}}, \quad x \mapsto x\vert_{M^{i+j}}. 
\end{equation}
Clearly, $\widetilde{p}_{ijk}(A_{k  j}) \subset A_{ij}$ since $D_{ij} \subset D_{k+j}$. 
We thus obtain well-defined projection maps 
\begin{equation}
\label{e:hor-trans-map}
p_{ijk} \colon  A_{k  j} \to  A_{ij}, 
\end{equation}
which extend the horizontal transition maps $\Sigma_{k  j} \to \Sigma_{ij}$ of the space-time inverse system $(\Sigma_{ij})_{i,j \in \N}$ associated with  $\tau \colon \Sigma  \to \Sigma$ and the memory set $M$. 

\begin{remark} 
\label{r:a-ij-not-subsystem}
In general, $(A_{ij})_{i,j \in \N}$ is not a subsystem of the   space-time inverse system 
$(A^{M^{i+j}})_{i, j \in \N}$ associated with $(A^G,\widetilde{\tau},M)$ (cf.~Remark \ref{r:notation-full-shift}).
 There is no trivial reason for $\widetilde{q}_{ij}(A_{i,j+1}) \subset A_{ij}$ unless $\Sigma$ is the full shift.   
\end{remark} 

The following lemma says that  each row of the system $(A_{ij}, p_{ijk})_{i,j,k \in \N}$ gives us an approximation 
of $\Sigma$.

\begin{lemma}
\label{l:sft-hor} 
For every $j \in \N$, there is a canonical bijection 
\[
\Psi_j  \colon \Sigma \to \varprojlim_{i \in \N} (A_{i j}, p_{i j k}).  
\]
\end{lemma}

\begin{proof} 
Since $\Sigma_{i j} \subset A_{ij}$ for all $i, j\in \N$, each $x \in \Sigma$ defines naturally an element $\Psi_j(x) = (x\vert_{M^{i+j}})_{i \in \N}\in \varprojlim_{i \in \N} (A_{ij}, p_{ijk})$. Conversely, 
let $(x_i)_{i \in \N} \in \varprojlim_{i \in \N} (A_{ij}, p_{ijk})$.
 Define $x \in A^G$ by setting, 
for each $g \in G$, $x (g) \coloneqq x_i(g)$ for 
any $i \in \N$ large enough such that $g \in M^{i + j}$.
The fact that the configuration $x \in A^G$ is well defined follows from (Mem1) and (Mem2).
Let $g \in G$.
Take $i$ large enough so that $g D \subset M^{i + j}$.
Then $g \in D_{i j}$ and $(g^{-1} x)\vert_D = (g^{-1} x_i)\vert_D \in P$ since $x_i \in A_{i j}$.
This shows that $x \in \Sigma$.  
\end{proof}

\begin{lemma} 
\label{l:sft-general}
For all $i, j \in \N$, we have   
\[ 
\Sigma_{ij}  \subset \bigcap_{k \geq i} p_{i j k}(A_{k j}). 
\]
\end{lemma} 

\begin{proof}
Let $y \in \Sigma$ and let $x= y\vert_{M^{i+j}} \in \Sigma_{ij}$. 
Let  $k \geq i$.  
Since  $y\vert_{M^{k+j}} \in \Sigma_{M^{k+j}} \subset A_{k j}$, 
it follows that 
\[
x = y\vert_{M^{i+j}}   = p_{ijk}(y\vert_{M^{k+j}}) \subset p_{ijk}(A_{k  j}). 
\]
As $y$ is arbitrary, the proof is finished. 
\end{proof}

% SECTION 6
\section{Algebraic subshifts of finite type}
\label{s:alg-space-time} 
Keeping the notations and hypotheses of Section~\ref{s:sft},   
we assume in this section that $A=V(K)$ and $P=W(K)$, 
where $V$ is an algebraic variety over an algebraically closed field $K$ 
and $W \subset V^D$ is an algebraic subvariety. 
Thus, $\Sigma=\Sigma(D, P) \subset A^G$ is an algebraic subshift of finite type. 
%Then $\Sigma$ is also an algebraic sofic subshift so that we can talk about the conditions (H1), (H2), (H3) in the Introduction. 
\par 
For all $i, j \in \N$, it is clear that  
$A_{ij}$ is a closed algebraic subset of $A^{M^{i+j}}$ 
since it is a finite intersection of sets of closed points of closed subvarieties of $V^{M^{i+j}}$: 
\begin{equation} 
\label{e:a-ij-closed}
A_{ij} =  \bigcap_{g \in D_{ij}} \pi_{ij,g}^{-1} (gW) (K). 
\end{equation}

Here, $\pi_{ij,g} \colon V^{M^{i+j}} \to V^{gD}$ is the projection induced by the inclusion 
$gD \subset M^{i+j}$ for $g \in D_{ij}$.  
The subset $gW \subset V^{gD}$ is defined as the image of $W$ under the isomorphism 
$V^D \simeq V^{gD}$ induced by the bijection $D \simeq gD$ given by $h \mapsto gh$ for every $h \in D$.  
\par
Observe that the maps $\pi_{ij,g}$ above and the transition maps of the inverse system 
$(\Sigma_{ij})_{i,j \in \N}$ 
are induced by morphisms of algebraic varieties.  
\par 
In this section, we consider the following conditions (C2) and (C3): 
 \begin{enumerate}
 \item [(C2)]
 $V$ is a complete $K$-algebraic variety;  
\item [(C3)]
 $V$ is a $K$-algebraic group and $W \subset V$ 
is an algebraic subgroup. 
\end{enumerate}
\begin{remark}
In the case (C3), note that the projections 
$p_{ijk} \colon  A_{k  j} \to  A_{ij}$ (cf.~\eqref{e:hor-trans-map}) 
are homomorphisms of algebraic groups. 
\end{remark}

\begin{proposition} 
\label{p:sigma-pro-constructible} 
With the above notations and hypotheses, suppose in addition that one of the conditions 
$(\mathrm{H1})$, $(\mathrm{C2})$, $(\mathrm{C3})$ is satisfied. 
Then, for each $i, j \in \N$, we have   
\begin{equation}
\label{e:pro-constructible-equality}
\Sigma_{ij} = \bigcap_{k \geq i} p_{ijk}(A_{k j}) 
\end{equation}
and  $\Sigma_{i j}$ is a countably-proconstructible subset of $A^{M^{i+j}}$.
Moreover, in the case $(\mathrm{C2})$ (resp.~$(\mathrm{C3})$),  
$\Sigma_{ij}$ is a complete subvariety (resp.~an algebraic subgroup) of $A^{M^{i+j}}$. 
\end{proposition}

\begin{proof} 
The inclusion $\Sigma_{ij} \subset \bigcap_{k \geq i} p_{ijk}(A_{k j})$ follows 
from Lemma \ref{l:sft-general}. 
\par 
Let $x \in \bigcap_{k \geq i} p_{ijk}(A_{k j}) \subset A^{M^{i+j}}$. 
We must show that $x$ can be extended to an element of $\Sigma$. 
Consider the following inverse system lying above $x$. 
Let $B_i=\{x\}$ and for each $k \geq i$, we set 
\begin{equation}
\label{e:b-ij}
B_{k} \coloneqq (p_{ijk})^{-1} (x) \subset A_{k j} ,    
\end{equation} 
which is a closed algebraic subset of $A^{M^{k+j}}$.  
Since $x \in  p_{ijk}(A_{k j})$, each set $B_{k}$ is nonempty. 
From \eqref{e:b-ij}, it is clear that for every $k \geq i$,  we have 
\[
\widetilde{p}_{kj}(B_{k+1} ) \subset B_k. 
\] 
By restricting the map $\widetilde{p}_{kj}$ to $B_{k+1}$, we have for each $k \geq i$ a well-defined algebraic map 
$\pi_k \colon B_{k+1} \to B_k$.  
Thus, we obtain an inverse subsystem $(B_k)_{k \geq i}$ of  
with transition maps $\pi_{nm} \colon B_m \to B_n$, where $m \geq n \geq i$, as compositions of the maps $\pi_k$. 
\par 
We claim that $\varprojlim_{k \geq i} B_k \neq \varnothing$.  
Indeed,  this follows from 
Lemma~\ref{l:inverse-limit-seq-const} if (H1) is satisfied  and from 
Lemma~\ref{l:inverse-limit-closed-im} and Remark~\ref{r:h1-h2} in the case of (C2). 
Suppose now that (C3) is satisfied. Since  $x \in  \bigcap_{k \geq i} p_{ijk}(A_{k j})$, 
there exists for each 
$k~\geq~i$ a point $z_k \in B_k$ such that $p_{ijk}(z_k)=x$. 
Let $V_k = \ker p_{ijk}$ be an algebraic subgroup of $A^{M^{k+j}}$ then clearly $B_k = z_k V_k$ 
where the group law is written multiplicatively. 
For all integers $m  \geq n \geq i$,  the map 
$\pi_{mn}$ is the restriction of a homomorphism of algebraic groups (cf.~Remark~\ref{r:full-system-homomorphism}). 
Therefore, $\pi_{nm}(B_m)$ is a translate of an algebraic subgroup of  $A^{M^{n+j}}$ 
and thus is Zariski closed in $B_n$ (cf.~Remark~\ref{r:h1-h2}). 
Hence, the claim follows, also in the case of (C3), from Lemma~\ref{l:inverse-limit-closed-im}. 
\par  
Therefore, we can find $(y_k)_{k \geq i} \in \varprojlim_{k \geq i} B_k$.   
Let $y \in A^{G}$ be defined as follows. Given $g \in G$,  set $y(g)= y_k(g)$ for any $k \geq i$  such that 
$g \in M^{k+j}$. Then $y$ is well-defined  by (Mem2). 
For each $g \in G$, choose $k \geq i$   so that $gD \subset M^{k+j}$. 
Then $(g^{-1}y)\vert_D = y_k\vert_{gD} \in W(K)$ which follows from  
the definition of $A_{kj}$ and since $y_k \in B_{kj} \subset A_{kj}$. 
Hence, $y \in \Sigma$. By construction, $x= y \vert_{M^{i+j}}$ and we deduce that 
$\bigcap_{k \geq i} q_{ijk}(A_{k j}) \subset \Sigma_{M^{i+j}}$. The proof of the formula 
\eqref{e:pro-constructible-equality} is completed. Thus, by Chevalley's theorem, 
$\Sigma_{i j}$ is a countably-proconstructible subset of $A^{M^{i+j}}$. 
\par
Finally, the last statement follows from \eqref{e:pro-constructible-equality} and Remark~\ref{r:h1-h2} 
and Noetherianity of the Zariski topology of $A^{M^{i+j}}$. Note that 
the sequence $(q_{ijk}(A_{k j}))_{k \geq i}$ is trivially a descending sequence.  
\end{proof}

\begin{corollary}
\label{c:sigma-pro-constructible} 
With the above notations and hypotheses, suppose that  
the condition $(\mathrm{H1})$ (resp.~$(\mathrm{C2})$, resp.~$(\mathrm{C3})$) is satisfied for $\Sigma$. 
Then, for each finite subset $E \subset G$, the restriction $\Sigma_E $  
is a countably-proconstructible subset (resp.~a complete subvariety, resp.~an algebraic subgroup) of $A^E$. 
\end{corollary}

\begin{proof}
Let $i, j \in \N$ be large enough so that $E \subset M^{i+j}$. 
Let $\pi \colon A^{M^{i+j}} \to A^E$ be the induced projection. It follows that 
$\Sigma_E = \pi (\Sigma_{ij})$. 
In the cases (C2) and (C3), Proposition~\ref{p:sigma-pro-constructible} and 
Remark~\ref{r:h1-h2} imply that 
$\Sigma_E$ is respectively a complete subvariety and an algebraic subgroup of  $A^E$. 
In the case (H1),  we find by Lemma \ref{l:pro-constructible} that
\begin{align}
\Sigma_E = \pi (\Sigma_{ij}) = \pi \left(\bigcap_{k \geq i} p_{ijk} (A_{kj}) \right)
= \bigcap_{n \in \N} \pi (p_{ijk} (A_{kj}) ). 
\end{align} 
Hence, $\Sigma_E$ is countably-proconstructible 
by Chevalley's theorem. The proof is completed. 
\end{proof}

%SECTION 7
\section{Algebraic sofic subshifts}
\label{s:alg-sofic-shift} 

Consider the following hypothesis without condition on cellular automata: 

\begin{enumerate} 
\item [$(\mathrm{\widehat{H3}})$]
$K$ is algebraically closed, $V$ is a $K$-algebraic group, and $\Sigma \subset A^G$ is an algebraic group sofic subshift. 
\end{enumerate}
We can now state the main local result for algebraic sofic subshifts.   

\begin{theorem}
\label{t:sofic-pro-constructible}
Let $V$ be an algebraic variety over a field $K$ and let $A= V(K)$.
Let $G$ be a finitely generated group and let $\Sigma \subset A^G$ be an algebraic sofic subshift.
Let $E$ be a finite subset of $G$.
Suppose that the condition $(\mathrm{H1})$ (resp.~$(\mathrm{H2})$, resp.~$(\mathrm{\widehat{H3}})$) is satisfied.  
Then  the restriction $\Sigma_E \subset A^E$  
is  a countably-proconstructible subset (resp.~a complete subvariety,
resp. an algebraic subgroup) of  $A^E$.
\end{theorem}
\begin{proof}
By hypothesis, there exist in the cases (H1) and (H2) 
an algebraic variety (resp. in the case $(\mathrm{\widehat{H3}})$ an algebraic group) $U$ over $K$, 
an algebraic (resp. algebraic group) cellular automaton $\tau' \colon B^G \to A^G$ where $B= U(K)$, 
and an algebraic (resp. algebraic group) subshift of finite type 
$\Sigma' \subset B^G$ such that $\Sigma = \tau' (\Sigma')$. 
Note that $U,V$ are complete varieties in the case (H2). 
Let $M$ be a memory set of $\tau'$. 
By Corollary~\ref{c:sigma-pro-constructible}, 
the set $\Sigma'_{ME}$ is countably-proconstructible. 
Hence, $\Sigma'_{ME} = \bigcap_{n \in \N} C_n$ where 
$(C_n)_{n \in \N}$ is some decreasing sequence of constructible subsets of $A^{ME}$. 
Let $\varphi \colon B^{ME} \to A^E$ be given by $\varphi(x)(g)= \tau'(y)(g)$ 
for every $x \in B^{ME}$, $g \in E$ and every $y \in B^G$ extending $x$. 
Then $\varphi$ is algebraic (cf.~\cite[Lemma~3.2]{ccp-2019}) and in the case $(\mathrm{\widehat{H3}})$, 
it is a homomorphism of algebraic groups (cf.~\cite[Lemma~3.4]{phung-2018}).  
In the case (H1), we can conclude by Chevalley's theorem since: 
\begin{align}
\Sigma_E = (\tau'(\Sigma'))_E = \varphi(\Sigma'_{ME}) 
= \varphi \left(\bigcap_{n \in \N} C_n \right) = \bigcap_{n \in \N} \varphi (C_n)
\end{align}
where the last equality follows from Lemma \ref{l:pro-constructible}. 
Finally, in the cases (H2) and $(\mathrm{\widehat{H3}})$, Corollary~\ref{c:sigma-pro-constructible} implies that 
$\Sigma_E   = \varphi (\Sigma'_{ME})$ is respectively a complete subvariety and  
an algebraic subgroup of  $A^E$. 
\end{proof}

% SECTION 8
\section{A closed mapping property and chain recurrent sets} 
\label{s:cip}
Using the space-time inverse system, we give a short proof of the 
following result saying that the image  of an algebraic 
sofic subshift under an algebraic cellular automaton is closed.
It extends the linear case in \cite[Theorem 4.1]{JPAA}  
\par 
Let $G$ be a group. 
Let $V_0, V_1$ be algebraic varieties over an algebraically field $K$. 
Let $A_0 = V_0(K)$ and let $A_1=V_1(K)$.  
Let 
$\tau \colon A_0^G \to A_1^G$ be an algebraic cellular automaton 
and let $\Sigma \subset A_0^G$ be an algebraic sofic subshift.  
\par 
Then $\Sigma$ is the image of some algebraic subshift of finite type $\Sigma' \subset B^G$ under 
an algebraic cellular automaton $\tau' \colon B^G \to A_0^G$, where 
$B$ is the set of  $K$-points of a $K$-algebraic variety $U$.    
\par 
To avoid notational confusion, we introduce in this section the following hypotheses similar to (H1), (H2), and (H3), $(\mathrm{\widehat{H3}})$. 
\begin{enumerate} 
\item [$(\mathrm{\widetilde{H1}})$]
$K$ is uncountable;  
\item [$(\mathrm{\widetilde{H2}})$]
$U, V_0$ are complete $K$-algebraic varieties; 
\item [$(\mathrm{\widetilde{H3}})$]
$U, V_0$, and $V_1$ are $K$-algebraic groups, 
$\Sigma' \subset B^G$ is an algebraic group subshift of finite type, 
$\tau' \colon  B^G \to A_0^G$ and 
$\tau \colon  A_0^G \to A_1^G$ are algebraic group cellular automata.
\end{enumerate}

\begin{theorem}
\label{t:closed-image}
With the above notations, 
if one of the conditions $(\mathrm{\widetilde{H1}})$, $(\mathrm{\widetilde{H2}})$, $(\mathrm{\widetilde{H3}})$ is satisfied, 
then $\tau(\Sigma)$ is closed in $A_1^G$.  
\end{theorem}

\begin{proof}
It is clear that, up to replacing $\tau$ by  the composition 
$\tau \circ \tau'$ and $\Sigma$ by $\Sigma'$, we can suppose without 
loss of generality that $\Sigma$ is an algebraic subshift of finite type. 
The hypotheses $(\mathrm{\widetilde{H2}})$, $(\mathrm{\widetilde{H3}})$ 
now become respectively: 
\begin{enumerate} 
\item [$(\mathrm{P2})$]
$V_0$ is a complete $K$-algebraic variety; 
\item [$(\mathrm{P3})$]
$V_0$, and $V_1$ are $K$-algebraic groups, 
$\Sigma \subset A_0^G$ is an algebraic group subshift of finite type, and 
$\tau \colon  A_0^G \to A_1^G$ is an algebraic group cellular automaton.
\end{enumerate}

Let $D \subset G$ be a defining memory set of $\Sigma$. 
Let $d \in A_1^G$ be in the closure of $\tau(\Sigma)$.  
We must show that $d \in \tau(\Sigma)$. 
\par 
Suppose first that $G$ is finitely generated. 
Let $M \subset G$ be a finite memory subset of $\tau$ containing $\{1_G\} \cup D$ 
which generates $G$ and satisfies $M=M^{-1}$. 
Consider the inverse system $(A_0^{M^i})_{i\in \N}$ 
whose transition maps 
$p _{ij}\colon A_0^{M^j} \to A_0^{M^i}$, where $0 \leq i \leq j$, are defined as 
the canonical projections induced by the inclusions $M^i \subset M^j$.  
For every $i \geq 1$, 
the induced map $q_i \colon A_0^{M^i} \to A_1^{M^{i-1}}$ is given as follows. 
For every $\sigma \in A_0^{M^i}$, 
we set  $q_i(\sigma) \coloneqq (\tau(x))\vert_{M^{i-1}}$ where  $x \in A_0^G$ is 
 any configuration that extends $\sigma$. 
For every $i \geq 1$, we define 
\[ 
Z_{i} \coloneqq   q_i^{-1}(d\vert_{M^{i-1}}) \cap \Sigma_{M^i}.  
\] 
\par 
Since $d$ belongs to the closure of $\tau(\Sigma)$ in $A_1^G$, 
it follows that $Z_i \neq \varnothing$ for every $i \geq 1$. 
By restricting the projections $p_{ij} \colon A^{M^j} \to A^{M^i}$ 
to the  $Z_j$'s, we obtain well-defined transition maps $\pi_{ij} \colon Z_j \to Z_i$, 
where $j \geq i \geq 1$, of the inverse system $(Z_i)_{i \geq 1}$. 
\par
It suffices to show that $\varprojlim_{i \geq 1} Z_i \neq \varnothing$
since, by construction of  $Z_i$ and $\Sigma_{ij}$ (see also
\cite[Lemma
2.1]{ccp-2019}), we have $\tau(c) = d$ for every $c \in \varprojlim_{i \geq 1}
Z_i \subset \varprojlim_{i \geq 1} \Sigma_{M^{i+1}} = \Sigma $ (by \eqref{e:row-limit} since $\Sigma$ is closed 
as it is a subshift of finite type).      
 \par 
Thanks to Theorem~\ref{t:sofic-pro-constructible},  the 
 conclusion follows by a direct application of  
 Lemma~\ref{l:inverse-limit-seq-const}, resp. 
 Lemma~\ref{l:inverse-limit-closed-im}, to the inverse system $(Z_i)_{i \in \N}$ 
 if $(\mathrm{\widetilde{H1}})$, resp. (P2), is satisfied.  
 Assume now that (P3) is satisfied.  
For each $i \geq 1$, choose $z_i \in Z_i$ and 
let $V_i \coloneqq \ker q_i  \cap \Sigma_{M^i}$ 
be an algebraic subgroup of $\Sigma_{M^i}$ (by Theorem~\ref{t:sofic-pro-constructible}). 
We have $Z_i= z_i V_i$. Hence (by Remark~\ref{r:full-system-homomorphism}), 
for $j \geq i \geq 1$,  
$\pi_{ij}(Z_j) $ 
is  a translate of an algebraic subgroup of $\Sigma_{M^i}$ and thus is Zariski closed in $Z_i$. 
Therefore, the case (P3) follows from Lemma  \ref{l:inverse-limit-closed-im}. 
\par
For a general group $G$,  consider a finite memory set $M$ of 
$\tau$ containing $\{1_G\} \cup D$ and such that $M=M^{-1}$. 
Let $H \subset G$ be the subgroup generated by $M$. 
As $\Sigma_H$ is clearly an algebraic (resp. in the case $(\mathrm{\widetilde{H3}})$ an algebraic group) 
subshift of finite type, 
the above discussion shows that $\tau_H(\Sigma_H)$ 
is closed in $A_1^G$ and so is $\tau(\Sigma)$ by Lemma  \ref{l:restriction-cip}. 
\end{proof}

\begin{corollary} 
\label{c:sofic-closed}
Let $G$ be a group.  
Let $V$ be an algebraic variety over a field $K$ and 
let $A \coloneqq V(K)$. 
Let $\Sigma \subset A^G$ be an algebraic sofic subshift. 
If one of the conditions $(\mathrm{H1})$, $(\mathrm{H2})$, $(\mathrm{\widehat{H3}})$ is satisfied, 
then $\Sigma$ is closed in $A^G$.  
\end{corollary}

\begin{proof}
It suffices to apply Theorem~\ref{t:closed-image} in the case $V_0=V_1$ 
to the identity map $\tau = \Id_{A^G}$ where $A = V_0(K) = V_1(K)$. 
\end{proof}

\begin{corollary} 
\label{c:cr-subset-omega}
With the notations and hypotheses as in Theorem~\ref{t:limit-set-not-empty}, we have  
$\Crec (\tau) \subset \Omega(\tau)$. 
\end{corollary}

\begin{proof}
By Proposition \ref{p:chain-recurrent}, we only need to check that 
$\tau^n(\Sigma)$ is closed in $A^G$ for every $n \geq 1$ and 
that $\tau$ is uniformly continuous. The first property follows from 
Theorem \ref{t:closed-image}. The second is a general property of cellular automata 
already mentioned in the Introduction.   
\end{proof}

%SECTION AJOUTÉE 
\section{Applications to backward orbits and limit sets}  
Thanks to the closedness property of algebraic sofic subshifts, we can establish the following key relation between 
inverse space-time systems, backward orbits and limit sets.   
\begin{theorem} 
\label{t:space-time-ls-alg}
Let $V$ be an algebraic variety over a field $K$ and let $A= V(K)$. 
Let $G$ be a finitely generated group and let $\Sigma \subset A^G$ 
be an algebraic sofic subshift.  
Let $\tau \colon \Sigma \to \Sigma$ be an algebraic cellular automaton.  
Assume that one of the conditions $(\mathrm{H1})$, 
$(\mathrm{H2})$, $(\mathrm{H3})$ is satisfied. 
Then, with the notations as in Section \ref{s:space-time-system}, 
we have a surjective map $\Phi \colon \varprojlim_{i,j \in \N} \Sigma_{ij}  \to \Omega(\tau)$. 
\end{theorem} 

\begin{proof}
By Corollary~\ref{c:sofic-closed}, the subshift $\Sigma$ is closed in $A^G$. 
Hence, $\varprojlim_{i,j \in \N} \Sigma_{ij}$ 
is the set of backward orbits of $\tau$ and 
we have a canonical map 
$\Phi \colon \varprojlim_{i,j \in \N} \Sigma_{ij}  \to \Omega(\tau)$ given in Lemma \ref{l:space-time-ls}. 
Now let $y_0 \in \Omega(\tau) \subset \Sigma$. 
We must show that there exists $x \in \varprojlim_{i,j \in \N} \Sigma_{ij} $ such that $\Phi(x)=y_0$.  
For every $i, j \in \N$, define a closed subset 
\[
B_{ij}  \coloneqq (q_{i 0} \circ \dots \circ q_{i,j-1})^{-1} (y_0\vert_{M^{i}}) \subset \Sigma_{ij}. 
\] 
By definition  of $\Omega(\tau)$, there exists for every $j \in \N$ 
an element $y_j \in \Sigma$ such that $\tau^j (y_j)=y_0$. 
Hence, it follows from the definition of the transition maps $q_{ik}$ 
and of $\Sigma_{ij}$ that $y_j\vert_{M^{i+j}} \in B_{ij}$. 
In particular, $B_{ij} \neq \varnothing$ for every $i, j \in \N$. 
By restricting the transition maps of the space-time inverse system $(\Sigma_{ij})_{i, j \in \N}$ 
to the sets $B_{ij}$, we obtain a well-defined inverse subsystem $(B_{ij})_{i, j \in \N}$.
\par 
We claim that $\varprojlim_{i, j \in \N} B_{ij} \neq \varnothing$. 
Indeed, by Theorem~\ref{t:sofic-pro-constructible}, 
the case (H1) is implied by Lemma~\ref{l:inverse-limit-seq-const}. 
In the case (H2), Theorem~\ref{t:sofic-pro-constructible} 
implies that $B_{ij}$ is a complete algebraic subvariety of $\Sigma_{ij}$ and thus of $A^{M^{i+j}}$.  
Hence, the case (H2) follows from Lemma~\ref{l:inverse-limit-closed-im} and Remark~\ref{r:h1-h2}. 
In the case (H3),  a similar argument as in the proof of Proposition~\ref{p:sigma-pro-constructible} 
shows that the transition maps of the system $(B_{ij})_{i, j \in \N}$ have Zariski closed images. 
Therefore, the case (H3) follows immediately from Lemma~\ref{l:inverse-limit-closed-im}.  
Thus, we can find   
\[
x  \in \varprojlim_{i, j \in \N} B_{ij}  \subset \varprojlim_{i, j \in \N} \Sigma_{ij}. 
\] 
It is clear from the constructions of the inverse system $(B_{ij})_{i, j \in \N}$ 
and of the map $\Phi$ (see the proof of Lemma \ref{l:space-time-ls}) 
that $\Phi(x)= y_0$. 
The proof of the lemma is completed. 
\end{proof}

\begin{corollary}
\label{c:omega-completely-invariant} 
With the notations and hypotheses as in Theorem \ref{t:limit-set-not-empty}, 
we have $\tau(\Omega(\tau))= \Omega(\tau)$. 
\end{corollary}

\begin{proof} 
Let $M\subset G$ be a finite subset 
containing $1_G$, a memory set of $\tau$ and a memory set of $\Sigma$ 
and  such that $M=M^{-1}$.  
 Let $H \subset G$ be the subgroup generated by $M$. 
Since $\tau = \prod_{c \in G/H} \tau_c$ and $\Omega(\tau)= \prod_{c \in G /H} \Omega(\tau_c)$ 
(cf. Lemma \ref{l:restriction-ls}), we can suppose without loss of generality that $G=H$. 
Let $x \in \Omega(\tau)$ then $x \in \tau^n(X)$ for every $n \geq 0$.  
Thus $\tau(x) \in \tau^{n+1}(X)$ for every $n \geq 0$ and it follows that $\tau(x) \in \Omega(\tau)$. 
Therefore, $\tau(\Omega(\tau)) \subset \Omega(\tau)$. 
For the converse inclusion, let $y \in \Omega(\tau)$. By Theorem~\ref{t:space-time-ls-alg}, 
there exists $x = (x_{ij}) \in \varprojlim_{i,j \in \N} \Sigma_{ij}$ such that $ \Phi(x)=y$. 
On the other hand, \eqref{e:backward-orbits} tells us that $\Phi^{-1}(y) \subset 
\varprojlim_{i,j \in \N} \Sigma_{ij}$ is the set of backward orbits of $y$ under $\tau$. 
Hence, we can find $z \in \Omega(\tau)$ such that $\tau(z)= y$. Thus,  
$\Omega(\tau) \subset \tau(\Omega(\tau))$ and the conclusion follows. 
\end{proof}

%SECTION 9
\section{Noetherianity of algebraic subshifts of finite type} 

The goal of this section is to establish the following characterization of 
algebraic subshifts of finite type by the descending chain property. 
It extends the linear version in \cite[Theorem 1.1 and Corollary 1.2]{JPAA}.
The proof is an application of Theorem~\ref{t:sofic-pro-constructible} 
combined with the construction of an inverse system analogous to the space-time inverse system.  
More precisely, we obtain: 

\begin{theorem} 
\label{t:noetheiran-alg-sft}
Let $G$ be a finitely generated group and let $V$ be an algebraic variety 
(resp. an algebraic group) over an algebraically closed field $K$.  
Let $A=V(K)$ and let $\Sigma \subset A^G$ be a subshift. 
Consider the following properties: 
\begin{enumerate} [\rm (a)] 
\item 
$\Sigma$ is a subshift of finite type; 
\item 
$\Sigma$ is an algebraic (resp. algebraic group) subshift of finite type; 
\item 
every descending sequence of  algebraic (resp. algebraic group) sofic subshifts of $A^G$ 
\begin{equation*} 
\Sigma_0 \supset \Sigma_1 \supset \cdots \supset \Sigma_n \supset \Sigma_{n+1} \supset  \cdots  
\end{equation*} 
such that $\bigcap_{n \geq 0} \Sigma_n = \Sigma$ eventually stabilizes. 
\end{enumerate}
Then we have $\mathrm{(b)}$$\implies$$\mathrm{(a)}$$\implies$$\mathrm{(c)}$. 
Moreover, if $\Sigma \subset A^G$ is an algebraic (resp. algebraic group) sofic subshift, 
then $\mathrm{(a)}$$\iff$$\mathrm{(b)}$$\iff$$\mathrm{(c)}$. 
 \end{theorem} 

\begin{proof} 
It is trivial that (b)$\implies$(a).   
Assume that  $\Sigma$ is a subshift of finite type. 
Hence, $\Sigma = \Sigma(D, W)$ where $D \subset G$ is finite and 
$W \subset A^D$ is some subset. 
Let $\Sigma_0 \supset \Sigma_1 \supset \cdots$ 
be a descending sequence of  algebraic (resp. algebraic group) sofic subshifts of $A^G$ 
whose intersection is $\Sigma$. 
Let $M \subset G$ be a finite generating subset containing $\{1_G\} \cup D$ and such that $M=M^{-1}$.  
Consider the inverse system $(X_{ij})_{i,j \in \N}$ defined by $X_{ij} \coloneqq (\Sigma_j)_{M^{i}} \subset A^{M^i}$. 
Remark that $X_{i,j+1} \subset X_{ij}$ since $\Sigma_{j+1} \subset \Sigma_j$ for all $i, j \in \N$. 
We define the unit transition maps $p_{ij} \colon X_{i+1, j} \to X_{ij}$ by 
$p_{ij}(x)= x\vert_{M^{i}}$ for every $x \in X_{i+1, j}$ 
and $q_{ij} \colon X_{i, j+1} \to X_{ij}$ simply as the inclusion maps. 
\par 
For all $i, j \in \N$, Theorem~\ref{t:sofic-pro-constructible} 
implies that every $X_{ij}$ is a complete variety (resp. an algebraic group) over $K$. 
By Noetherianity of the Zariski topology, the decreasing sequence 
$(X_{0j})_{j \in \N}$ of algebraic closed subsets of $A^{M}$ eventually stabilizes, 
say, $X_{0j}= X_{0m}$ for all $j \in \N$ for some $m \in \N$. 
Let $W' \coloneqq X_{0m}$ then $\Sigma' \coloneqq \Sigma(M^m, W')$ is an algebraic 
(resp. algebraic group)  subshift of finite type.  
It is clear that $\Sigma_m \subset \Sigma'$  
and hence $\Sigma \subset \Sigma'$. We shall prove the converse inclusion. 
\par 
Let $w \in  W'$. 
We construct an inverse subsystem $(Z_{ij})_{i \geq m, j \geq 0}$ of $(X_{ij})_{i \geq m, j \geq 0}$ as follows. 
For $i \geq m$, 
let $Z_{i0} \coloneqq \{ x \in  X_{i0} \colon x\vert_{M^m} = w\}$ 
which is clearly an algebraic closed subvariety (resp. a translate of an algebraic subgroup) of $X_{i0}$. 
For $i \geq m, j \geq 0$, we define an algebraic closed subvariety 
(resp. a translate of an algebraic subgroup (by Theorem~\ref{t:sofic-pro-constructible})) of $X_{ij}$ as follows: 
\[
Z_{ij} \coloneqq ( q_{i0}  \circ  \dots \circ q_{i, j-1})^{-1} (Z_{i0}) \subset X_{ij}. 
\]
\par 
The transition maps of $(Z_{ij})_{i \geq m, j  \geq 0}$ are well-defined as the restrictions of the transition maps of 
the system $(X_{ij})_{i \geq m, j \geq 0}$. 
These transition maps have Zariski closed images (by Remark~\ref{r:h1-h2}). 
\par 
By our construction, each $Z_{ij}$ is  clearly nonempty. 
Hence, Lemma~\ref{l:inverse-limit-closed-im} implies that 
there exists $x = (x_{ij})_{i \geq m,j \geq 0} \in \varprojlim Z_{ij}$. 
Let $y \in A^G$ be defined by $y(g)= x_{i0}(g)$ for every $g \in G$ and 
any large enough $i\geq m$ such that $g \in M^{i}$. 
Observe that $x_{ij} = x_{ik}$ for every $i \geq m$ and $0 \leq j \leq k$  
since the vertical transition maps $X_{ik} \to X_{ij}$ are simply inclusions. 
Consequently, for every $n \in \N$, we have $y \in \Sigma_n$ by \eqref{e:row-limit} since $\Sigma_n$ is closed in $A^G$ (cf.~Corollary~\ref{c:sofic-closed}). 
Hence $y \in \Sigma$. 
By construction, $y\vert_{M^m} = w$. 
Since $w$ was arbitrary, this shows that $W' \subset \Sigma_{M^m}$. 
Hence, $\Sigma' = \Sigma(M^m, W') \subset \Sigma(M^m, \Sigma_{M^m}) = \Sigma$. 
The last equality follows from Lemma~\ref{l:base-change-sft} as $D \subset M^m$.  
Therefore, $\Sigma' = \Sigma $ and  
$\Sigma_n= \Sigma$ for all $n \geq m$. This proves that (a)$\implies$(c).  
\par
Suppose now that $\Sigma \subset A^G$ is an algebraic (resp. algebraic group)
sofic subshift which is not a subshift of finite type. 
Let $M \subset G$ be a finite generating subset containing $\{1_G\}$ such that $M=M^{-1}$. 
For every $n \in \N$, consider $W_{n} \coloneqq \Sigma_{M^n}$  
(as in Section~\ref{s:space-time-system}). 
Theorem~\ref{t:sofic-pro-constructible} tells us that $W_n$ is a complete algebraic subvariety 
(resp. an algebraic subgroup) of $A^{M^n}$. 
Set $\Sigma_n \coloneqq \Sigma(M^n, W_n)$ for every $n \in \N$ then 
$\Sigma_n$ is an algebraic (resp. algebraic group) subshift of finite type.  
As $(\Sigma_{M^{n+1}})_{M^n} = \Sigma_{M^n}$, 
it is clear that $\Sigma \subset \Sigma_{n+1} \subset \Sigma_n$ for every $n \in \N$. 
We claim that $\Sigma = \bigcap_{n \in \N} \Sigma_n$. 
Indeed, we only need to prove that $\bigcap_{n \in \N} \Sigma_n \subset \Sigma$. 
Let $x \in \bigcap_{n \in \N} \Sigma_n$. Then by definition of $\Sigma_n$, 
we find that $x\vert_{M^n} \in W_n= \Sigma_{M^n}$ for every $n \in \N$. 
Thus, since $\Sigma$ is closed (cf.~Corollary~\ref{c:sofic-closed}), $x \in \varprojlim_{n \in \N} \Sigma_{M^n} = \Sigma$ (cf.~\eqref{e:row-limit}) and hence 
$\bigcap_{n \in \N} \Sigma_n \subset \Sigma$. 
However, the descending sequence $(\Sigma_n)_{n \in \N}$ cannot stabilize since otherwise the subshift  
$\Sigma$ would be  of finite type. 
This shows that (c)$\implies$(a) 
if $\Sigma$ is an algebraic (resp. algebraic group)
sofic subshift. The proof is complete.  
\end{proof}

% SECTION 10 
\section{Proof of Theorem~\ref{t:limit-set-not-empty}} 
 For (i), we know that $\Omega(\tau)$ is $G$-invariant by the $G$-equivariance of $\tau$.
On the other hand,
as the set of algebraic cellular automata over $\Sigma$ is  
closed under the composition of maps (cf.~\cite[Proposition~3.3]{ccp-2019} 
for the case of full shifts, the general case is proved similarly), 
the map $\tau^n \colon \Sigma  \to \Sigma$ is an algebraic cellular automaton for every $n \geq 1$. 
It then follows from Theorem~\ref{t:closed-image} 
that $\tau^n(\Sigma)$ is closed in $A^G$ for every $n \geq 1$ 
and thus $\Omega(\tau) = \bigcap_{n \geq 1} \tau^n(\Sigma)$ is also closed in $A^G$. 
This shows that $\Omega(\tau)$ is a closed subshift of $A^G$ and (i) is proved. 
\par 
For (v), let $z \in \Sigma \cap \Fix(H)$ for some subgroup $H \subset G$. 
Let $F \subset G$ be the subgroup generated by a finite subset $M$ containing $\{1_G\}$ and memory sets of 
$\tau$, $\Sigma$,  
and such that $M=M^{-1}$. 
Note that $z\vert_F$ is fixed by the subgroup $R \coloneqq F \cap H$ of $F$. 
Consider the space-time inverse system $(\Sigma_{ij})_{i, j \in \N}$ associated with 
the restriction $\tau_F$ and $M$ as in Definition~\ref{d:space-time-system}. 
Keep the notations in Section \ref{s:space-time-system}. For all $i, j \in \N$, let 
\begin{equation*}
\label{e:delta-r}
\Delta_{ij} \coloneqq \{ p \in A^{M^{i+j}} \colon p(u)=p(v) \mbox{ for all } u, v \in M^{i+j} \mbox{ with } uv^{-1} \in R \} 
\end{equation*}
be the restriction to $M^{i+j}$ of $R$-fixed points in $A^H$. 
Clearly, $\Delta_{ij}$ is respectively 
a closed subvariety, a complete algebraic subvariety,  
an algebraic subgroup of $A^{M^{i+j}}$ in the cases (H1), (H2), (H3).  
Define also $Z_{ij} \coloneqq \Sigma_{ij} \cap \Delta_{ij} \subset A^{M^{i+j}}$. 
\par
Note that $\tau_F$ sends $R$-fixed points to $R$-fixed points.   
Hence, by restricting the transition maps to the sets $Z_{ij}$, 
we obtain a well-defined inverse subsystem of 
$(\Sigma_{ij})_{i,j \in \N}$. 
Theorem~\ref{t:sofic-pro-constructible} implies that respectively in each case 
(H1), (H2), (H3), the set $Z_{ij} \subset A^{M^{i+j}}$ is a countably-proconstructible subset, 
a complete algebraic subvariety, an algebraic subgroup. 
Each $Z_{ij}$ is nonempty since it contains $z\vert_{M^{i+j}}$. 
Hence, Lemma~\ref{l:inverse-limit-seq-const} and Lemma~\ref{l:inverse-limit-closed-im} imply that  
there exists $x \in \varprojlim_{i,j \in \N} Z_{ij} \subset  \varprojlim_{i,j \in \N} \Sigma_{ij}$. 
By Lemma~\ref{l:space-time-ls}, we obtain $y= \Phi(x) \in \Omega(\tau_F) \subset A^H$. 
By our construction, $y$ is fixed by $R$. 
By Lemma~\ref{l:restriction-ls}.(i), $\Omega(\tau) = \Omega(\tau_F)^{G/F}$.
Thus  $y$ induces a configuration of $\Omega(\tau)$ fixed by $H$. 
This proves (v). 
 \par 
To finish the proof of Theorem~\ref{t:limit-set-not-empty}, note that (ii) follows from   
Corollary~\ref{c:omega-completely-invariant} and (iii) follows from the general property 
$\Per(\tau) \subset \Rec(\tau) \subset \NW(\tau) \subset \Crec(\tau)$ and from the inclusion 
$\Crec(\tau) \subset \Omega(\tau) $ proved in Corollary~\ref{c:cr-subset-omega}. 
Finally, in the case (H2) (resp. in the case (H3)), 
(iv) is a direct consequence of the implication (a)$\implies$(c) in Theorem~\ref{t:noetheiran-alg-sft} applied 
to $\Omega(\tau)$ and the decreasing sequence of algebraic (resp. algebraic group) sofic subshifts 
$\tau(\Sigma) \supset \tau^2(\Sigma) \supset \cdots \supset \tau^n(\Sigma) \supset \tau^{n+1}(\Sigma) \supset \cdots$ 
which satisfies $\bigcap_{n \geq 1} \tau^n(\Sigma) \eqqcolon \Omega(\tau)$  by definition. 
\qed

%SECTION 11
\section{Proof of Theorem \ref{t:char-nilpotent-alg-ca}}

 It is clear that (a)$\implies$(b). 
For the converse implication, suppose that $\Omega(\tau)= \{x_0\}$ for some $x_0 \in \Sigma$. 
As $\Omega(\tau)$ is $G$-invariant, there exists en element denoted by $0 \in A$ 
such that $x_0(g)= 0$ for every $g \in G$, i.e., $x_0=0^G$.  
Let $M\subset G$ be a finite subset containing 
a memory set of $\tau$ and a memory set of $\Sigma$ such that $1_G \in M$ and $M=M^{-1}$. 
Let $H$ be the subgroup generated by $M$ and consider the restriction $\tau_H$. 
By Lemma~\ref{l:restriction-ls}, we deduce that $\Omega(\tau_H)$ must be a singleton as well and 
$\tau$ is nilpotent if $\tau_H$ is. Thus, up to replacing $G$ by $H$, we can suppose that 
$G$ is generated by $M$. 
\par 
We construct an inverse subsystem $(\Sigma^*_{ij})_{i, j \in \N}$ 
of the space-time inverse system $(\Sigma_{ij})_{i, j \in \N}$ 
associated with $\tau$ and the memory set $M$  
(cf. Definition \ref{d:space-time-system}) as follows. 
Let $\Sigma^*_{i0} \coloneqq \Sigma_{i0} \setminus \{x \in \Sigma_{i0} \colon x(1_G)=0\}$ 
for every  $i \geq 0$. For all $i \geq 0$ and $j \geq 1$, we  define  
\[
\Sigma^*_{ij}= (q_{i 0} \circ \dots \circ q_{i, j-1})^{-1}(\Sigma_{i0}^*). 
\] 
The unit transition maps $q^*_{ij} \colon \Sigma^*_{i, j+1} \to \Sigma^*_{i j}$ 
and $p^*_{ij} \colon \Sigma^*_{i+1, j} \to \Sigma^*_{ij}$ of the inverse subsystem $(\Sigma^*_{ij})_{i, j \in \N}$
are defined respectively by the restrictions of the transition maps 
$q_{ij}$ and $p_{ij}$ of $(\Sigma_{ij})_{i, j \in \N}$. 
\par 
Assume on the contrary that $\tau$ is not nilpotent. 
We claim that $\Sigma^*_{ij} \neq \varnothing$ for all $i, j \in \N$. 
Otherwise, $\Sigma^*_{ij}= \varnothing$ for some $i, j \in \N$. 
If $j=0$ then $\Sigma^*_{i0}= \varnothing$, that is, $x(1_G) = 0$ for
all $x \in \Sigma_{i0}$, and since $\Sigma$ is $G$-invariant and
$1_G \in M^{i+j}$, we deduce that $\Sigma = \{0^G\}$.
Hence, $\tau$ is trivially nilpotent and we arrive at a contradiction. 
Thus, $j \geq 1$ and by definition  of $\Sigma_{ij}$, we have for every $x \in \Sigma_{ij}$ that 
\[
(q_{i 0} \circ \dots \circ q_{i, j-1})(x)(1_G)=0.
\] 
Since  $\tau^{i+j}$ is $G$-equivariant, 
it follows that $\tau^{i+j}(x)=0^G$ for every $x \in \Sigma$, 
which contradicts the assumption that $\tau$ is not nilpotent. 
This proves the claim, i.e., 
$\Sigma^*_{ij} \neq \varnothing$ for all $i, j \in \N$.  
We are going to show that 
\begin{align}
\label{e:inverse-space-time-nonempty}
\varprojlim_{i,j \in \N} \Sigma^{*}_{ij} \neq \varnothing. 
\end{align}
Indeed, \eqref{e:inverse-space-time-nonempty} is a direct application of 
Theorem~\ref{t:sofic-pro-constructible} and Lemma \ref{l:inverse-limit-seq-const} 
to the inverse system $(\Sigma^*_{ij})_{i, j \in \N}$ in the case (H1). 
For the cases (H2) and (H3), observe that 
for every $(i, j) \prec(k, l)$ in $\N^2$, we have 
\begin{align}
\label{e-complete-space-time}  
Z\coloneqq F_{(i,j), (k,l)} (\Sigma^*_{kl}) = F_{(i,j), (k,l)}(\Sigma_{kl}) \cap \Sigma^*_{ij},  
\end{align}
where $F_{(i,j), (k,l)} \colon \Sigma_{kl} \to \Sigma_{ij}$ 
is the transition map of the inverse system $(\Sigma_{ij})_{i, j \in \N}$. 
Indeed, by definition  of $\Sigma^*_{kl}$ and $\Sigma^*_{ij}$, 
and using the equality $F_{(i,0),(k,l)} = F_{(i,0),(i,j)} \circ F_{(i,j), (k,l)}$, 
we see that  
\begin{align*}
F_{(i,j), (k,l)} (\Sigma^*_{k,l}) & = F_{(i,j), (k,l)} (\Sigma_{k,l} \setminus F^{-1}_{(i,0),(k,l)}(\Sigma_{i,0} \setminus \Sigma^*_{i,0}) ) \\
& \supset  F_{(i,j), (k,l)} (\Sigma_{kl})  \setminus F_{(i,j), (k,l)}(F^{-1}_{(i,0),(k,l)}(\Sigma_{i0} \setminus A^*_{i0}) ) \\ 
& \supset  F_{(i,j), (k,l)} (\Sigma_{kl})  \setminus  F^{-1}_{(i,0),(i,j)}(\Sigma_{i0} \setminus \Sigma^*_{i0}) \\
& =  F_{(i,j), (k,l)} (\Sigma_{kl})  \setminus (\Sigma_{ij} \setminus \Sigma^*_{ij}) \\
& =  F_{(i,j), (k,l)} (\Sigma_{kl})  \cap \Sigma^*_{ij}. 
\end{align*} 
But clearly $F_{(i,j), (k,l)} (\Sigma^*_{kl})  \subset F_{(i,j), (k,l)} (\Sigma_{kl})  \cap \Sigma^*_{ij}$, and 
\eqref{e-complete-space-time} is proved. 
\par 
In the cases (H2) and (H3), the set $F_{(i,j), (k,l)}(\Sigma_{kl})$ is closed in $\Sigma_{ij}$ 
by Remark~\ref{r:h1-h2}, Remark~\ref{r:full-system-homomorphism}, and Theorem~\ref{t:sofic-pro-constructible}. 
We infer from \eqref{e-complete-space-time} 
that $F_{(i,j), (k,l)} (\Sigma^*_{kl})$ is a Zariski closed subset of $\Sigma^*_{ij}$. 
Therefore, $\varprojlim_{i, j \in \N} \Sigma^{*}_{ij} \neq \varnothing$ results from 
Lemma~\ref{l:inverse-limit-closed-im} 
and \eqref{e:inverse-space-time-nonempty} is proved in all cases. 
\par 
We can thus choose $x =(x_{ij})_{i, j \in \N} \in \varprojlim_{i,j \in \N} \Sigma^{*}_{ij}$. 
Let $\Phi \colon \varprojlim_{i,j \in \N} \Sigma_{ij} \to \Omega(\tau)$ 
be the map given in Theorem~\ref{t:space-time-ls-alg}. 
As $\varprojlim_{i,j \in \N} \Sigma^{*}_{ij} \subset \varprojlim_{i,j \in \N} \Sigma_{ij}$, 
we obtain $y_0 = \Phi (x) \in \Omega(\tau)$. 
As $y_0(1_G) = x_{00}(1_G)$ by definition of $\Phi$ 
and as $x_{00}(1_G) \neq 0$ since $x_{00} \in \Sigma^*_{00}$, 
we deduce that $\Omega(\tau) \neq \{0^G\}$. 
This contradiction shows that (b)$\implies$(a).  
\qed

%SECTION 12 
\section{Nilpotency over finite alphabets}  
 The following theorem strengthens and extends to any infinite group some results
established for full shifts over $G = \Z$  by
Culik, Pachl, and Yu \cite[Theorem 3.5]{culik-limit-sets-1989}
and by Guillon and Richard \cite[Corollary 4]{guillon-richard-2008}. 
\par 
Suppose that $X$ is a topological space equipped with a continuous action of a group $G$. 
One says that the dynamical  system $(X,G)$ is 
\emph{topologically mixing} if
for each pair of nonempty open subsets $U$ and $V$ of $X$ there exists a
finite subset
$F \subset G$ such that $U \cap gV \neq \varnothing$ for all $g \in G
\setminus F$.
Given a group $G$ and a finite set $A$, a closed subshift $\Sigma
\subset A^G$ is said to be \emph{topologically mixing}
provided $(\Sigma,G)$ is topologically mixing. 
If $(X,G)$ is a topologically mixing dynamical system and $f
\colon X \to X$ is a continuous $G$-equivariant map, then the factor system $(f(X), G)$ is
also topologically mixing.

\begin{theorem}
\label{t:finit-limit-set}
Let $G$ be an infinite group, let $A$ be a finite set, and let $\Sigma
\subset A^G$ be a nonempty
topologically mixing subshift of sub-finite-type (e.g., $\Sigma = A^G$, or, if $G$ is finitely generated, $\Sigma$ is of
finite type).
Let $\tau \colon \Sigma \to \Sigma$ be a cellular automaton.
Then the following conditions are equivalent:

\begin{enumerate}[{\rm (a)}]
\item $\tau$ is nilpotent;
\item the limit set $\Omega(\tau)$ is reduced to a single configuration; 
\item the limit set $\Omega(\tau)$ is finite.
\end{enumerate}
If $G$ is finitely generated, then the above conditions are equivalent
to
\begin{enumerate}[{\rm (d)}]
\item the limit set $\Omega(\tau)$ consists only of periodic
configurations.
\end{enumerate}
\end{theorem}

Before starting the proof of the above theorem, we present a 
preliminary lemma. The result is probably well known, but since we could not find any reference,
we include a proof for the sake of completeness.

\begin{lemma}
\label{l:claim1-finite}
Let $G$ be a finitely generated group, let $A$ be a set, and let
$\Sigma \subset A^G$ be a finite subshift.
Then $\Sigma$ is of finite type.
\end{lemma}

Roughly, the idea is simple. Every configuration $x \in \Sigma$ 
has a finite orbit, equivalently, its stabilizer $H_x = \Stab_G(x)$ is of finite
index in $G$. Since the intersection of finitely many finite-index
subgroups is of finite index, the group $H \coloneqq \bigcap_{x \in \Sigma} H_x$ 
is of finite index in $G$. Moreover, by the Poincar\'e lemma,
there exists a finite index normal subgroup $K \subset H$. This way, we
can embed $\Sigma$ into $A^{G/K}$ (cf. \cite[Proposition 1.3.7]{book}). 
As $G/K$ is finite, it follows that $\Sigma$ is of finite type. The
proof below is a detailed and self-contained version of the above idea.
See \cite[Proposition 2.4]{JPAA} for a linear version (where ``finite''
becomes ``finite-dimensional'').

\begin{proof}[Proof of Lemma \ref{l:claim1-finite}]
Let $S \subset G$ be a finite generating subset of $G$.
After replacing $S$ by $S \cup S^{-1} \cup \{1_G\}$, we can assume that
$S = S^{-1}$ and $1_G \in S$.
Then, given any element $g \in G$, there exist $n \in \N$ and $s_1, s_2,
\ldots, s_n \in S$ such that
$g = s_1s_2 \cdots s_n$. The minimal $n \in \N$ in such an expression of
$g$ is  the $S$-length of $g$,
denoted by $\ell_S(g)$.

For all distinct $x,y \in \Sigma$ we can find $g = g_{x,y} \in G$ such
that $x(g) \neq y(g)$.
Then the finite set $D_0 \coloneqq \{g_{x,y}: x, y \in \Sigma \mbox{
s.t. } x \neq y\} \subset G$
satisfies  
\begin{equation}
\label{e:def-D-0}
x\vert_{D_0} = y\vert_{D_0} \mbox{ implies } x = y \mbox{ for all } x
\in \Sigma.
\end{equation}
Let us show that $\Sigma = \Sigma(D,P)$ for $D = SD_0$ and $P \coloneqq
\{x\vert_{D}: x \in \Sigma\} \subset A^D$.
By definition, we have 
\begin{multline}
\label{e:sigma-D-P}
\Sigma(D,P)  = \{x \in A^G : \mbox{ for all }  g \in G \mbox{ there
exists } x_g \in \Sigma \mbox{ such that } \\
(g^{-1} x)(d) = x_g(d) \text{  for all  } d \in D \}.
\end{multline}

Note that the element $x_g \in \Sigma$ in \eqref{e:sigma-D-P} is
uniquely defined by
$x \in \Sigma(D,P)$ and $g \in G$, since $D \supset D_0$ so that
$x_g\vert_D = x_g'\vert_D$ infers $x_g = x_g'$ by \eqref{e:def-D-0}.

We clearly have $\Sigma \subset \Sigma(D,P)$, since $\Sigma$ is
$G$-invariant.
\par 
For the converse inclusion, suppose that $x \in \Sigma(D,P)$ and
let us show that $x = x_{1_G} \in \Sigma$.
We prove by induction on the $S$-length of $g$ that
\begin{equation}
\label{e:g-1-G}
x_g = g^{-1}x_{1_G}
\end{equation}
for all $g \in G$. If $\ell_S(g) = 0$, then $g = 1_G$ and
\eqref{e:g-1-G} holds trivially.
Suppose now that $\ell_S(g) = n$ and let $s \in S$. Given $d_0 \in D_0$
we have, on the one hand
$x(gsd_0) = (g^{-1}x)(sd_0) = x_g(sd_0) = s^{-1}x_g(d_0)$, and, on the
other hand,
$x(gsd_0) = ((gs)^{-1}x)(d_0) = x_{gs}(d_0)$. This shows that
$(s^{-1}x_g)\vert_{D_0} = x_{gs}\vert_{D_0}$.
Since $s^{-1}x_g$ and $x_{gs}$ both belong to $\Sigma$, we deduce from
\eqref{e:def-D-0} that $s^{-1}x_g = x_{gs}$.
By induction, we have $x_g = g^{-1}x_{1_G}$ so that $x_{gs} =
(gs)^{-1}x_{1_G}$. This proves \eqref{e:g-1-G}.
 From \eqref{e:g-1-G} we obtain, for every $g \in G$,
\[
x(g) = (g^{-1}x)(1_G) = x_g(1_G) = g^{-1}x_{1_G}(1_G) = x_{1_G}(g).
\]
This shows that $x = x_{1_G} \in \Sigma$.
\end{proof}

\begin{proof}[Proof of Theorem \ref{t:finit-limit-set}]
The equivalence (a)$\iff$(b) follows from 
Theorem~\ref{t:char-nilpotent-alg-ca}. 
The implication (b)$\implies$(c) is obvious.
\par 
Suppose now that $\Omega(\tau)$ is finite. Let $M \subset G$ be a finite
subset which serves as a memory set for both $\tau$ and $\Sigma$,
and denote by $H \subset G$ the subgroup it generates. Let $\tau_H
\colon \Sigma_H \to \Sigma_H$ denote the corresponding restriction
cellular automaton. It follows from Lemma \ref{l:restriction-ls} that
$\Omega(\tau) = \Omega(\tau_H)^{G/H}$.
If $G$ is not finitely generated, then $G/H$ is infinite and necessarily
$\Omega(\tau_H)$ and therefore $\Omega(\tau)$ must consist of
a single element, as $\Omega(\tau)$ is nonempty (cf.\
Theorem~\ref{t:limit-set-not-empty}). This proves the implication (c)
$\implies$ (b) for $G$ not finitely generated.

If $G$ is finitely generated, it follows from Lemma~\ref{l:claim1-finite} 
that $\Omega(\tau)$ is s subshift of finite type.
Since $A$ is finite, the characterization of subshifts of finite type in Theorem
\ref{t:noetheiran-alg-sft} can be applied to the sequence
\[
\Sigma \supset \tau(\Sigma) \supset \tau^2(\Sigma) \supset \cdots
\supset \Omega(\tau) = \bigcap_{n \in \N} \tau^n(\Sigma),
\]
and implies that there exists $n_0 \in \N$ such that $\Omega(\tau) =
\tau^{n_0}(\Sigma)$. 
Therefore, $\Omega(\tau)$ is a factor of $\Sigma$.
Since $\Sigma$ is topologically mixing, so is $\Omega(\tau)$. 
Now let $x, y \in \Omega(\tau)$. As $\Omega(\tau)$ is finite and Hausdorff,  
$\{x\}$ and $\{y\}$  are open in $\Omega(\tau)$. 
Thus, by topological mixing of $\Omega(\tau)$, there exists a finite
subset $F \subset G$ such that $x = gy$ for all $g \in G
\setminus F$. Since 
$\Omega(\tau)$ is finite, the stabilizer $H$ of $y$ in $G$ is an infinite subgroup
of $G$. It follows that $H \cap (G \setminus F) \neq \varnothing$. Taking $g \in H \cap (G
\setminus F)$ yields $x = gy = y$. Hence, $\Omega(\tau)$ is a singleton and  
this concludes the proof of
the implication (c)$\implies$(b).
\par 
Finally, suppose that $G$ is finitely generated.
As any finite $G$-invariant subset of $A^G$ necessarily consists only
of periodic configurations, we have (c)$\implies$(d).
The reverse implication follows from 
the finiteness of closed subshifts containing only periodic configurations proved    
in \cite[Theorem~5.8]{ballier-phd-2009} and in
\cite[Theorem~1.4]{meyerovitch-salo-2019} (see also 
\cite[Theorem~3.8]{ballier-durand-jeandel-2008} for the case $G=
\Z^2$). Note that since $A^G$ is compact and $\tau$ is continuous, 
$\Omega(\tau)$ is  closed in $A^G$. 
\end{proof}

%SECTION 13 
\section{Proof of Theorem~\ref{t:char-nilpotent-finite-alg-ca}} 
By Corollary \ref{c:sofic-closed}, we know that $\Sigma$ is closed in $A^G$. 
The equivalence (a)$\iff$(b) thus results from Proposition~\ref{p:carct-nilp-ca}.  
It is trivial that (a)$\implies$(c)$\implies$(d). 
For the implications (d)$\implies$(a) and (c)$\implies$(a),  
let $M \subset G$ be a finite 
subset containing the memory sets of both $\tau$ and $\Sigma$ 
such that $1_G \in M$ and $M=M^{-1}$. 
Let $H$ be the subgroup of $G$ generated by $M$. Let $\tau_H
\colon \Sigma_H \to \Sigma_H$ denote the restriction
cellular automaton. 
By Lemma~\ref{l:restriction-ls}.(i), we have 
$\Omega(\tau) = \Omega(\tau_H)^{G/H}$. 
Thus, if $\Omega(\tau)$ is finite then so is $\tau(\tau_H)$. Likewise, if (d) holds for $\tau$ 
then $\{ x(1_G) \colon x \in \Omega(\tau_H) \}$ is finite 
and $\Omega(\tau_H)$ consists of periodic configurations as well. 
On the other hand, $\tau$ is nilpotent if $\tau_H$ is nilpotent by Lemma~\ref{l:restriction-ls}.(ii).  
Therefore, up to replacing $G$ by $H$, we can assume that $G$ is finitely generated by $M$.   
It then suffices to show that (d)$\implies$(a) as we already know that (c)$\implies$(d). 
 \par  
Assume that (d) holds. Then $T \coloneqq \{ x(1_G) \colon x \in \Omega(\tau) \}$  
is finite. 
As $\Omega(\tau)$ is $G$-invariant,  $x(g)\in T$ for every $x \in \Omega(\tau)$ and $g \in G$.   
\par 
Let $(\Sigma_{ij})_{i, j \in \N}$ 
be the space-time inverse system associated with $\tau$ and the memory set $M$.  
We set $\Sigma^*_{i0} \coloneqq \Sigma_{i0} \setminus \{x \in \Sigma_{i0} \colon x(1_G)\in T\}$ 
for every $i \geq 0$  and define  for every $i \geq 0$ and $j \geq 1$: 
\[
\Sigma^*_{ij}= (q_{i 0} \circ \dots \circ q_{i, j-1})^{-1}(\Sigma_{i0}^*) \subset \Sigma_{ij}.   
\] 
The unit transition maps $p^*_{ij} \colon \Sigma^*_{i+1, j} \to \Sigma^*_{ij}$
and $q^*_{ij} \colon \Sigma^*_{i, j+1} \to \Sigma^*_{i j}$ 
are respectively the restrictions of the transition maps $p_{ij}$ and $q_{ij}$   
of the system $(\Sigma_{ij})_{i, j \in \N}$. 
\par
Suppose first that $\Sigma^*_{ij} \neq \varnothing$ for all $i, j \in \N$.  
Then exactly as in the proof of Theorem  \ref{t:char-nilpotent-alg-ca}, 
there exists $x =(x_{ij})_{i, j \in \N} \in \varprojlim_{i,j \in \N} \Sigma^{*}_{ij}$ and we obtain  
$y_0 = \Phi (x) \in \Omega(\tau)$ with $y_0(1_G) = x_{00}(1_G)$. 
But $x_{00}(1_G) \notin T$ because $x_{00} \in \Sigma^*_{00}$, 
we find that $\Omega(\tau) \not \subset T^G$ which is a contradiction.  
\par 
Therefore, we must have $\Sigma^*_{ij}= \varnothing$ for some $i, j \in \N$. 
If $j=0$ then $\Sigma^*_{i0}= \varnothing$ and $x(1_G) \in T$ for
all $x \in \Sigma_{i0}$. We deduce that $A^G \subset T^G$ and thus $A \subset T$ is finite. 
As $G$ is infinite and $\Omega(\tau)$ contains only periodic configurations, 
Theorem~\ref{t:finit-limit-set} implies that $\tau$ is nilpotent. 
If $j \geq 1$ then by definition  of $\Sigma_{ij}$, we have for every $x \in \Sigma_{ij}$ that 
\[
(q_{i 0} \circ \dots \circ q_{i, j-1}) (x)(1_G) \in T.
\] 
Hence, as $\tau^{j}$ is $G$-equivariant, we deduce that $\tau^{j}(x) \in T^G$ for every $x \in A^G$. 
Thus,  the restriction $\sigma \coloneqq \tau^{j}\vert_{T^G} \colon T^G \to T^G$ is a well-defined cellular automaton. 
As a subset of  $\Omega(\tau)$, the set $\Omega(\sigma)$ also consists of periodic configurations.  
We deduce from Theorem~\ref{t:finit-limit-set} that $\sigma$ is nilpotent, say, 
$\sigma^m(x) = x_0$ for all $x \in T^G$ for some $m \in \N$ and $x_0 \in A^G$.  
It follows that $\tau^{(m+1)j}(x)= \sigma^m(\tau^{j} (x) ) =x_0$ for all $x \in A^G$. 
We conclude that $\tau$ is nilpotent. The proof of the theorem is completed.   
 \qed 
 
 % SECTION 14
\section{Counter-examples}

The following example (c.f.~\cite[Example~5.1]{csc-algebraic} and \cite[Example~8.1]{ccp-2019}) 
shows that  Theorem~\ref{t:closed-image} and  Assertions~(i) and (iii)  of Theorem~\ref{t:limit-set-not-empty}  become false
if we remove the hypothesis that the ground field $K$ is algebraically closed.

\begin{example}
\label{ex:limit-set-not-closed}
Let $G\coloneqq \Z$ be the additive group of integers and let 
 $V \coloneqq \Spec(\R[t])=\A^1_{\R}$ denote 
the affine line over $\R$.
Then $A \coloneqq V(\R) = \R$. 
Consider the cellular automaton  $\tau \colon \R^{\Z}  \to \R^{\Z}$
with memory set   $M \coloneqq \{0,1\}\subset G$
and associated local defining map $\mu \colon \R^M \to \R$
defined by
$\mu(p) = p(1) - p(0)^2$  for all $p  \in \R^M$.
Clearly, $\tau$ is an algebraic cellular automaton over $(G,V,\R)$. 
Indeed,  $\mu$ is  
induced by the algebraic morphism  $f\colon V^2 \to V$ associated with  the morphism of $\R$-algebras
\begin{align*}
\R[t] & \to \R[t_0,t_1] \\
  t & \mapsto  t_1-t_0^2. 
\end{align*}
Note that $\tau \colon \R^{\Z}\to \R^{\Z}$ is given by
\[
\tau(c)(n) = c(n+1) - c(n)^2 \quad \text{for all } c\in \R^{\Z} \text{ and } n \in \Z.
\]

  \begin{claim}
 \label{claim:omega-not-shift}
The limit set  $\Omega(\tau)$ is a dense non-closed subset of $\R^{\Z}$. 
In particular, $\Omega(\tau)$ is not a closed subshift of $\R^{\Z}$. 
\end{claim} 

\begin{proof}
Let $c \in \R^{\Z}$ and let $F \subset \Z$ be a finite subset.
Choose an integer $m \in \Z$ such that $F\subset [m,\infty)$
and consider the configuration $d \in \R^{\Z}$ defined by
$d(n) \coloneqq  0$ if $n < m$ and $d(n) \coloneqq  c(n)$ if $n \geq m$.
For each $k \in \N$, define by induction on $k$ a configuration $d_k \in \R^{\Z}$
in the following way.
We first take $d_0 = d$.
Then, assuming that the configuration $d_k$ has been defined,
we define the configuration $d_{k + 1}$, using induction on $n$,
by $d_{k + 1}(n) \coloneqq 0$ if $n \leq  m$ and
$d_{k+1}(n+1) \coloneqq d_k(n) + d_{k + 1}(n)^2$ if $n \geq m$.
Clearly, $\tau(d_{k + 1}) = d_k$  so that
$d = d_0 = \tau^k(d_k)$ for all $k \in \N$.    
  Therefore $d \in \Omega(\tau)$. 
Since  $c$ and $d$ coincide on $[m,\infty)$ and hence on $F$,
this shows that $c$ is in the closure of $\Omega(\tau)$. 
Thus $\Omega(\tau)$ is dense in $\R^{\Z}$.
\par 
In \cite[Example~5.1]{csc-algebraic} and \cite[Example~8.1]{ccp-2019}, it is shown that 
$\im(\tau)$ is not closed of $R^\Z$ and 
the constant configuration $e\in \R^{\Z}$, defined by $e(n) \coloneqq 1$ for all $n\in \Z$, 
 does not belong to $\im(\tau)$. 
This implies that   $e \notin \Omega(\tau)$.
As $\Omega(\tau)$ is dense in $\R^{\Z}$,
we deduce that $\Omega(\tau)$ is not closed in $\R^{\Z}$.
\end{proof}

Remark that $\im(\tau)$ is an algebraic sofic subshift of $\R^\Z$ since it is the image of the full shift $\R^\Z$ under the 
algebraic cellular automaton $\tau$. 
Thus, an algebraic sofic subshift may fail to be closed in the ambient full shift. 
\par 
For every integer $n \geq 1$,   the set $M_n \coloneqq \{ 0, \dots, n \}$ is a memory set for $\tau^n$. 
Let $\mu_n \colon \R^{M_n} \to \R$ denote the associated local defining map. 
We shall use the fact  that  for each $n \geq 1$, there exists a polynomial $\nu_n \in \R[t_0, \dots, t_{n-1}]$ such that 
for every $p \in \R^{M_n}$, 
\begin{equation}
\label{e:exemple-NW}
\mu_n(p) = p(n) + \nu_n(p(0), \dots, p(n-1)).
\end{equation}
This fact can be  proved by an easy induction.
For $n=1$, we have $\mu_1(p) = \mu(p) =  p(1) - p(0)^2$ for every $p \in \R^{M_1}$ 
so that we can take $\nu_1(t_0)= - t_0^2$. 
Suppose now that the assertion holds for some $n\geq 1$. 
Let $c \in \R^\Z$ and let $d = \tau^{n}(c)$. 
By the induction hypothesis, we have that $d(0)= \mu_n(c(0), \dots, c(n))$ and 
\[
d(1) = \mu_n(c(1), \dots, c(n+1)) 
= c(n+1) + \nu_n(c(1), \dots, c(n)). 
\]
Therefore, we get 
\begin{align*} 
\tau^{n+1}(c)(0) & = \tau(\tau^{n}(c))(0) = \tau(d)(0) = d(1) - d(0)^2 \\ 
& = c(n+1) + \nu_k(c(1), \dots, c(n)) - \mu_n(c(0), \dots, c(k))^2 \\
& = c(n+1) + \nu_{n+1}(c(0), \dots, c(n)), 
\end{align*}
where $\nu_{n+1} \in \R[t_0, \dots, t_{n}]$ is given by the formula 
\[
\nu_{n+1} (t_0, \dots, t_{n}) \coloneqq \nu_n(t_1, \dots, t_n) - \mu_n(t_0, \dots, t_n)^2. 
\] 
Thus, for every 
$p \in \R^{M_{n+1}}$, 
\[
\mu_{n+1}(p) = p(n+1) + \nu_{n+1}(p(0), \dots, p(n)),  
\] 
and the assertion follows by induction.

\begin{claim}
\label{claim:induction-rec}
For every configuration $c \in \R^{\Z}$ and any integer $n \geq 1$, there exists $d \in \R^{\Z}$ such that  
$d(k) = c(k)$ for all  $k \leq 0$ and $\tau^{n}(d)(k)=c(k)$ for all $k \geq -n+1$.  
\end{claim}

\begin{proof}
Let $c \in \R^{\Z}$. 
We define $d \in \R^{\Z}$ by 
\[
d(k) = c(k) \quad \mbox{if } k \leq 0 , \\ 
\]
and inductively for $k \geq   1$ by 
\begin{equation} 
\label{e:exemple-NW-x}
d(  k) \coloneqq  c(k - n) - \nu_{n} (d(k -n), \dots, d(k-1)). 
\end{equation} 
By applying~\eqref{e:exemple-NW} and~\eqref{e:exemple-NW-x}, we obtain, for every $k \geq -n+1$,  
\begin{align*}
\tau^{n}(d)(k) 
& = \mu_{n}(d(k), \dots, d(n+k)) \\
& = d(n+k) + \nu_{n}(d(k), \dots, d(k+n-1)) \\ 
& = c(k),
\end{align*}
and the claim is proved.  
\end{proof} 

\begin{claim}
\label{claim:rec-dense}
The set $\Rec(\tau)$ is a dense non-closed subset of  $\R^{\Z}$.
In particular, $\Rec(\tau)$ is not a closed subshift of $\R^{\Z}$. 
\end{claim}

\begin{proof} 
Let $c \in \R^{\Z}$. 
For each $n_0 \geq 1$, define by induction on $n \geq n_0$ a configuration $d_n \in \R^{\Z}$
in the following way. 
Let $d_{n_0} = c$.
Then, assuming that the configuration $d_n$ has been defined,
we can choose by Claim~\ref{claim:induction-rec} and the $\Z$-equivariance of $\tau$ 
a configuration $d_{n + 1}$ satisfying $d_{n+1}(k) = d_n(k)$ for $k \leq 2 \cdot 3^{n}$ 
and  $\tau^{3^{n+1}}(d_{n+1})(k) = d_n(k)$ for $k \geq -3^n+1$.  
\par
Hence, we can define $d \in \R^{\Z}$ by setting $d(k)= d_n(k)$ for any $n \geq n_0$ such that $k \leq 2 \cdot 3^{n}$. 
Let $n \geq n_0$. 
Remark that $M_{3^{n+1}}$ is a memory set of $\tau^{3^{n+1}}$ 
and $d(k)= d_{n+1}(k)$ for $k \leq 2 \cdot 3^{n+1}$. Hence, for $-3^n +1 \leq  k \leq 3^n$ 
so that in particular $3^{n+1} + k \leq 2 \cdot 3^{n+1}$, we have 
\begin{equation} 
\tau^{3^{n+1}}(d)(k) = \tau^{3^{n+1}}(d_{n+1})(k) = d_n(k) =d(k).    
\end{equation} 
Since this holds for all $n \geq n_0$ and as every finite subset is contained in 
$\{ -3^n+1 , \dots,  3^n\}$ for any large enough $n$, we deduce that 
$d \in \Rec(\tau)$. 
\par 
It is clear from the construction that 
$d_n(k)= c(k)$ for every $n \geq n_0$ and $k \leq 2 \cdot 3^{n_0}$. 
Thus $d(k) = c(k)$ for every $k \leq 2 \cdot 3^{n_0}$. As $n_0 \geq 1$ 
is arbitrary, it follows that every $c \in \R^{\Z}$ belongs to the closure of $\Rec(\tau)$. 
In other words, $\Rec(\tau)$ is dense in $\R^{\Z}$.
\par
The set $\Rec(\tau)$ is not closed in $\R^{\Z}$.
Indeed, the configuration $e \in \R^{\Z}$, given by $e(k) = 1$ for all $k \in \Z$,
does not belong to $\Rec(\tau)$ since $\tau^n(e)(0) = 0 \not= c(0)$ for all $n \geq 1$.
\end{proof}

\begin{claim}
\label{claim:nw-cr-all}
One has $\NW(\tau) = \Crec(\tau) = \R^{\Z}$.
\end{claim}

\begin{proof}
By Claim~\ref{claim:rec-dense}, the set $\Rec(\tau)$ is dense in $\R^{\Z}$.
Since $\NW(\tau)$ and $\Crec(\tau)$ are closed in $\R^{\Z}$ and contain $\Rec(\tau)$, we deduce that 
$\NW(\tau) = \Crec(\tau) = \R^{\Z}$. 
\end{proof}

\begin{claim}
One has $\Rec(\tau) \not\subset \Omega(\tau)$ and $\Omega(\tau) \not \subset \Rec(\tau)$. 
\end{claim}

\begin{proof} 
By the proof of Claim~\ref{claim:omega-not-shift}, we know that 
the configuration $c \in \R^{\Z}$, given by $c(k)=0$ if $k \leq - 1$ and $c(k)=1$ if $k \geq 0$, 
belongs to  $\Omega(\tau)$. 
However, as $\tau^n(c)(0)=0 \neq c(0)$ for every $n \geq 1$, it follows that $c \notin \Rec(\tau)$.
\par 
On the other hand, reconsider the configuration $e \in \R^{\Z}$ given by $e(k)=1$ for every $k \in \Z$. 
The proof of Claim~\ref{claim:rec-dense} actually shows that there exists $d \in \Rec(\tau)$ such that 
$d(k)= e(k)=1$ for all $k \leq 0$. 
Suppose that there exists $b \in \R^{\Z}$ such that 
$\tau(b)= d$.
 Then $b(k+1) = 1 + b(k)^2$ for all $k \leq 0$. 
 Thus $1 \leq b(k) \leq b(k+1)$ for all $k \leq 0$, so that the limit $t \coloneqq  \lim_{k \to - \infty} b(k)$ exists and is finite. 
 By passing to the limit in the relation 
$b(k+1) = 1 + b(k)^2$, we find that $t=1+t^2$, which is a contradiction as $t$ must be real. 
This shows that $d \notin \tau(\R^{\Z})$ and thus $d \notin \Omega(\tau)$. The proof is completed. 
\end{proof}
\end{example}

\begin{remark} 
Consider the complex version of Example~\ref{ex:limit-set-not-closed}, that is, 
let $\tau_{\C} \colon \C^{\Z}  \to \C^{\Z}$ be the algebraic cellular automaton over 
$(\Z, \A^1_{\C},\C)$ 
with memory set $M = \{0,1\}\subset \Z$ 
and associated local defining map $\mu_{\C} \colon \C^M \to \C$ 
defined by 
$\mu_{\C}(p) = p(1) - p(0)^2$ for all $p  \in \C^M$. 
\par 
Then the same proofs as in Claim~\ref{claim:rec-dense} and Claim~\ref{claim:nw-cr-all} 
show that $\Rec(\tau_{\C})$ is a dense non-closed subset of  $\C^{\Z}$ and that 
$\NW(\tau_{\C}) = \Crec(\tau_{\C}) = \C^{\Z}$. 
By applying Theorem~\ref{t:limit-set-not-empty}.(ii), 
we deduce that $\Omega(\tau_{\C}) = \C^{\Z}$, i.e., $\tau_{\C}$ is surjective, 
which can also be easily checked by a direct verification. 
\end{remark}

The following example shows that Assertion (v) of Theorem~\ref{t:limit-set-not-empty} becomes false if we remove the hypothesis that the ground field $K$ is algebraically closed.

\begin{example}
\label{ex:alg-ca-empty-ls}
Let $G$ be a group and let  $V\coloneqq \Spec(\R[t])=\A^1_{\R}$ denote  the affine line over $\R$.
Consider the algebraic morphism  $f \colon V \to V$ given by $t \mapsto t^2 + 1$.
Take $A \coloneqq V(\R) = \R$ and let $\tau \colon A^G \to A^G$ denote the  cellular automaton with memory set $M \coloneqq  \{1_G\}$ and associated local defining map
$\mu \colon A^M = A \to A$ given by $a \mapsto a^2 + 1$.
The cellular automaton  $\tau$ is algebraic since $\mu$ is induced by $f$ but its limit set
$\Omega(\tau)$ is clearly empty. 
Remark also that $\tau$ is not stable since otherwise $\Omega(\tau)$ would be nonempty. 
   \end{example}

The following example shows that Theorem~\ref{t:char-nilpotent-alg-ca} becomes false if we remove the hypothesis that the ground field $K$ is algebraically closed.

\begin{example}
\label{ex:alg-not-nil-ca-ls-single}
Let $G$ be a group and let $V \coloneqq \Proj^1_{\R}$ denote  the projective line over $\R$.
Consider the algebraic morphism 
$f \colon V \to V$ given by $(x \colon y) \mapsto (x^2 + y^2 \colon y^2)$.
Take $A \coloneqq V(\R) = \Proj^1(\R) = \R \cup \{\infty\}$ and let $\tau \colon A^G \to A^G$ denote the  cellular automaton with memory set $M \coloneqq  \{1_G\}$ and associated local defining map
$\mu \colon A^M = A \to A$ given by $a \mapsto a^2 + 1$.
The cellular automaton  $\tau$ is algebraic since $\mu$ is induced by $f$.
Clearly, the limit set $\Omega(\tau)$ is reduced to the constant configuration $g  \mapsto \infty$ but $\tau$ is  not nilpotent.
\end{example} 

\section{Generalizations} 
\label{s:extension}

Using basic properties of proper morphisms, 
it is not hard to see that all the results for the (H2) case  (resp. for $(\mathrm{\widetilde{H2}})$ in Theorem~\ref{t:closed-image}) 
remain valid if $V$ (resp. $V_0$) is assumed to be separated (and not necessarily complete). 
For this, it suffices to remark that 
images of morphisms from a complete algebraic variety to a separated algebraic variety (cf.~\cite[Section~3.3.1]{liu-book}) 
are Zariski closed complete subvarieties  
(cf.~\cite[Section~3.3.2]{liu-book}). 
This leads us to the following definition.   
\begin{definition}
Let $G$ be a group and let $V$ be a separated algebraic variety over a field $K$. 
Let $A \coloneqq V(K)$. 
A subset $\Sigma \subset A^G$ is called a \emph{complete algebraic sofic subshift} 
if it is the image of an algebraic subshift of finite type $\Sigma' \subset B^G$, 
where $B=U(K)$ and $U$ is a complete $K$-algebraic variety,
under an algebraic cellular automaton $\tau' \colon B^G \to A^G$.  
\end{definition}

With the above definition,  
Theorem \ref{t:noetheiran-alg-sft} can also be extended as follows without any changes in the proof.   

\begin{theorem} 
\label{t:general-descent-complete}
Let $G$ be a finitely generated group. Let $V$ be a separated algebraic variety 
over an algebraically closed field $K$.  
Let $A=V(K)$ and let $\Sigma \subset A^G$ be a complete algebraic sofic subshift. 
Then following are equivalent:  
\begin{enumerate} [\rm (a)] 
\item 
$\Sigma$ is a subshift of finite type; 
\item 
$\Sigma$ is an algebraic subshift of finite type; 
\item 
every descending sequence of  algebraic sofic subshifts of $A^G$ 
\begin{equation*} 
\Sigma_0 \supset \Sigma_1 \supset \cdots \supset \Sigma_n \supset \Sigma_{n+1} \supset  \cdots  
\end{equation*} 
such that $\bigcap_{n \geq 0} \Sigma_n = \Sigma$ eventually stabilizes. 
\end{enumerate} 
 \end{theorem} 
\par 
Now, let $G$ be a group and let $V$ be an algebraic variety over a field $K$. 
Let $A=V(K)$ and let $\Sigma \subset A^G$ be a subset.  
 
\begin{definition} 
$\Sigma \subset A^G$ is called a \emph{countably-proconstructible subshift of finite type} (CPSFT) 
if there exist a finite subset $D \subset G$ and a subset $W \subset V^D$ 
which is the complement in $V^D$ of a countable number of constructible subsets (cf.~Section~\ref{s:countably-pro-constructible}), 
such that $\Sigma = \Sigma(D,W(K))$. 
Similarly,  $\Sigma \subset A^G$ is a \emph{countably-proconstructible sofic subshift} (CPS subshift) 
if it is the image of a CPSFT under an algebraic cellular automaton with range $A^G$. 
 \end{definition}
 
Our proofs actually show that  
Theorem~\ref{t:limit-set-not-empty} (except the point (iv)), Theorem~\ref{t:char-nilpotent-alg-ca} (resp. Theorem~\ref{t:char-nilpotent-finite-alg-ca})  
still hold if we replace (H1), (H2), (H3) and the assumption $\Sigma \subset A^G$ being an algebraic sofic subshift 
(resp. a topologically mixing algebraic sofic subshift)    
by a more general hypothesis: 
\begin{enumerate}[\rm (H)] 
\item 
$K$ is an uncountable algebraically closed field and 
$\Sigma \subset A^G$ is a $\CPS$ subshift  
(resp. topologically mixing $\CPS$ subshift)    
and $\tau \colon \Sigma \to \Sigma$ is an algebraic cellular automaton. 
\end{enumerate} 
In fact, it can be directly checked from our proofs 
that results for the (H1) case in Section~\ref{s:alg-space-time} (resp. Section~\ref{s:alg-sofic-shift} 
and Section~\ref{s:cip}) remain valid if we assume that $K$ is an uncountable algebraically closed field  
and $\Sigma$ is a $\CPSFT$ (resp. a $\CPS$ subshift).  
\par 
We now introduce a nontrivial class of nonempty $\CPSFT$ (cf.~Theorem~\ref{t:nonempty-full-cpsft}). 
\begin{definition} 
\label{d:cpsft}
Let $G$ be a group. 
Let $V$ be an algebraic variety over a field $K$ 
and let $A= V(K)$. 
A subshift $\Sigma \subset A^G$ is called a \emph{full $\CPSFT$} if there exist a finite subset $D \subset G$ and 
a subset $W = V^D \setminus ( \bigcup_{n \in \N} U_n)$ where each $U_n \subset V^D$ is a constructible subset 
satisfying $\dim U_n < \dim V^D$, such that $\Sigma = \Sigma(D, W(K))$. 
Here, $\dim Z$ denotes the Krull dimension of a constructible subset $Z$ (see for example \cite{ccp-goe-2020}). 
\end{definition} 
Remark that if $V$ is finite, i.e., $\dim V=0$, the conditions $\dim U_n < \dim V^D$ 
imply that $U_n = \varnothing$ for every $n \in \N$ thus $W=V^D$. 
Hence, when the alphabet is finite, the only full $\CPSFT$ is the full shift.  
\begin{example}  
\label{e:example-generalization}
If $G= \Z$, $A =  \C$, $D= \{0, 1\} \subset \Z$, $W = \C^D \setminus E$ where 
$E \subset \C^D \simeq \C^2$ is any countable union of complex algebraic curves and points, then 
$\Sigma' = \Sigma(D, W) \subset \C^\Z$ is a \emph{nonempty} full CPSFT (by Theorem~\ref{t:nonempty-full-cpsft} below). 
Let $\tau' \colon \C^\Z \to \C^\Z$ be given by 
$\tau'(x)(n) = x(n)^2 - x(n+1) +1$ for every $x \in \C^\Z, n \in \Z$, then $\Sigma \coloneqq \tau'(\Sigma')$ 
is a nonempty \emph{closed} CPS subshift of $\C^\Z$ 
(by~Theorem~\ref{t:closed-image} which is true under the condition (H)). 
Note that $\tau \coloneqq \tau'\vert_{\Sigma} \colon \Sigma \to \Sigma$ is an algebraic cellular automaton.  
\end{example} 
\begin{theorem} 
\label{t:nonempty-full-cpsft}
Let $G$ be a group. 
Let $V$ be a nonempty algebraic variety over an uncountable algebraically closed field $K$ 
and let $A= V(K)$.   
Then every full $\CPSFT$ $\Sigma \subset A^G$ is nonempty.
\end{theorem} 

\begin{proof} 
We write $\Sigma = \Sigma(D, W(K))$ for some finite subset $D \subset G$ and 
$W = V^D \setminus ( \bigcup_{n \in \N} U_n)$ where $U_n \subset V^D$, $n \in \N$, is a constructible subset 
such that $\dim U_n < \dim V^D$. 
In particular, $W$ is a countably-proconstructible subset of $V^D$.  
Suppose first that $G$ is finitely generated and let the notations be as in Section~\ref{s:sft}. 
Then same proof for the case (H1) of Proposition~\ref{p:sigma-pro-constructible} actually implies that 
$\Sigma_{ij} = \bigcap_{k \geq i} p_{ijk}(A_{kj})$ for $i, j \in \N$, 
where 
$
A_{ij} =  \bigcap_{g \in D_{ij}} \pi_{ij,g}^{-1} (gW) (K) \subset A^{M^{i+j}} 
$
(cf.~\eqref{e:a-ij-closed}). 
Note that $D_{ij}$ is finite and $g W \simeq W$ for all $g \in G$. 
It follows immediately that $A_{ij}$ is also a complement of a countable 
number of constructible subsets $Z_n$ such that $\dim Z_n < \dim A^{M^{i+j}}$ for every $n \in \N$. 
Hence, for every 
finite subset $I \subset \N$,  the constructible set $\bigcap_{n \in I} (A^{M^{i+j}} \setminus Z_n) \neq \varnothing$ 
by the dimensional reason. By Lemma~\ref{l:inverse-limit-seq-const}, 
we deduce that $A_{ij} =  \bigcap_{n \in \N} (A^{M^{i+j}} \setminus Z_n) \neq \varnothing$ for every $i, j \in \N$. 
Always by Lemma~\ref{l:inverse-limit-seq-const},  $\Sigma_{ij} = \bigcap_{k \geq i} p_{ijk}(A_{kj}) \neq \varnothing$ 
for all $i, j \in \N$ and thus $ \varprojlim_{i \in \N} \Sigma_{ij} \neq \varnothing$. 
Finally, the bijection $\Sigma \simeq \varprojlim_{i \in \N} \Sigma_{ij}$ (cf. \eqref{e:row-limit}) 
implies that $\Sigma \neq \varnothing$. 
\par 
For an arbitrary group $G$, let $H$ be the subgroup generated by $D$. 
Then by Lemma~\ref{l:factorization-sigma}, we have a factorization 
$\Sigma = \prod_{c \in G/H} \Sigma_c$ where the sets $\Sigma_c$ are pairwise homeomorphic. 
By the above paragraph, we know that $\Sigma_H \neq \varnothing$ and therefore 
$\Sigma \neq \varnothing$. 
\end{proof}
Theorem \ref{t:nonempty-full-cpsft} serves as a motivation for the notion of full CPSFT as we see in the following comparison with 
the finite alphabet case. It is well known that for $G= \Z^d$, $d \geq 2$, and for a finite set $A$ of cardinality at least 2, 
it is algorithmically undecidable whether the subshift of finite type $\Sigma(D,P) \subset A^G$ is nonempty 
for a given finite subset $D \in G$ and a given subset $ P \subset A^D$. 
This is known as the domino problem (cf.~\cite{aubrun-barbieri-sablik-2017}, \cite{berger-book-1966}, \cite{robinson-1971}; see
also the recent \cite{Bar-Salo}, where a notion of ``simulation'' for labelled graphs is introduced and applied to the domino problem for the Cayley graph of the lamplighter group and, more generally, to Diestel-Leader graphs). 

%SECTION 15
\appendix
\section{} 
\subsection{Limit sets and Nilpotency of general maps}
Given a set $X$, recall that a map $f \colon X \to X$ is \emph{pointwise nilpotent}
if there exists $x_0 \in X$ such that
for every $x \in X$, there exists an integer $n_0 \geq 1$ such that $f^n(x) = x_0$ for all $n \geq n_0$.
Such an $x_0$ is then the unique fixed point of $f$ and is called the \emph{terminal point} of the pointwise nilpotent map $f$.
Clearly, if $f$ is nilpotent then it is  pointwise nilpotent
and the terminal point of $f$ as a nilpotent map coincides with its terminal point as a pointwise nilpotent map.
Moreover, if $f$ is pointwise nilpotent, then its limit set is reduced to its terminal point.
When the set $X$ is finite, the three conditions (i) $f$ is nilpotent, (ii) $f$ is pointwise nilpotent,
and (iii) the limit set of $f$ is a singleton, are all equivalent.
This becomes false when $X$ is infinite.
Actually, we have the following.

\begin{lemma}
\label{l:non-nil-limit-one-pt}
Let $X$ be an infinite set.
Then the following hold:
\begin{enumerate}[\rm (i)]
\item
there exists a map $f \colon X \to X$ such that $\Omega(f) = \varnothing$;
\item
there exists a  map $f \colon X \to X$ which is not pointwise nilpotent (and hence not nilpotent)
such that $\Omega(f)$ is a singleton;
\item
 there exists a map $f \colon X \to X$ such that  $f(\Omega(f)) \subsetneqq \Omega(f)$;
\item
there exists a surjective (and hence non-nilpotent)  pointwise nilpotent map $f \colon X \to X$.
\end{enumerate}
\end{lemma}

\begin{proof}
(i)
Since $X$ is infinite,
there exists a bijective map
$\psi \colon \N \times X \to X$.
Then the map $f \colon X \to X$ defined by
$f \coloneqq \psi \circ g \circ \psi^{-1}$, where $g \colon \N \times X \to \N \times X$
is given by $g(n,x) = (n + 1,x)$ for all $(n,x) \in \N \times X$,
satisfies $\Omega(f) = \Omega(g) = \varnothing$.
This shows (i).
\par
(ii)
Let $\widehat{\N} \coloneqq  \N \cup \{\infty\}$.
Since $X$ is infinite, there exists an injective map $\varphi \colon \widehat{\N} \to X$ that is not surjective.
Then the map $f \colon X \to X$ defined by $f(\varphi(n)) = \varphi(n + 1)$
for all $n \in \N$ and $f(x) = \varphi(\infty)$ for all $x \in X \setminus \varphi(\N)$
satisfies $\Omega(f) = \{\varphi(\infty)\}$ but  is clearly not pointwise nilpotent.
\par
\par
(iii)
Consider, for each $n \geq 1$, the set $I_n \coloneqq \{0,1,\dots,n\}$ and the map $g_n \colon I_n \to I_n$ given by $g_n(k) \coloneqq k - 1$ if $k \geq 1$ and $g_n(0) = 0$. 
Let $Y$ be the set obtained by taking disjoint copies of the sets $I_n$, $n \geq 1$, and identifying all copies of $0$ in a single point $y_0$ and all copies of $1$ in a single point $y_1 \not= y_0$. 
Then the maps $g_n$ induce a well defined quotient map $g \colon Y \to Y$.
Clearly, $\Omega(g) = \{y_0,y_1\}$ while $g(\Omega(g)) = \{y_0\}$.
As $X$ is infinite, the set $Y$ can be regarded as a subset of  $X$.
Then the map $f \colon X \to X$, defined by $f(x) = g(x)$ if $x \in Y$ and $f(x) = x$ otherwise,
satisfies $\Omega(f) = \{y_0,y_1\} \cup (X \setminus Y)$ while
$f(\Omega(f)) = \{y_0\} \cup (X \setminus Y) \subsetneqq \Omega(f)$.  
\par
(iv)
Choose a point $x_0 \in X$ and a bijective map $\xi \colon \N \times X \to X \setminus \{x_0\}$.
Then the map $f \colon X \to X$, defined by
$f(\xi(n,x)) = \xi(n - 1,x)$ if  $n \geq 1$
and $f(x_0) = f(\xi(0,x)) = x_0$ for all $x \in X$,
is clearly surjective and pointwise nilpotent (with terminal point $x_0$).
\end{proof}

\subsection{Limit sets and nilpotency of general  cellular automata}

\begin{proposition}
\label{p:ca-non-nil-limit-pt}
Let $A$ be an infinite set and let $G$ be a group.
Then the following hold:
\begin{enumerate}[\rm (i)]
\item
there exists a cellular automaton $\tau \colon A^G \to A^G$ with $\Omega(\tau) = \varnothing$;
\item
there exists a non-nilpotent cellular automaton  $\tau \colon A^G  \to A^G$
such that $\Omega(\tau)$ is reduced to a single configuration;
\item 
there exists a cellular automaton $\tau \colon A^G \to A^G$ which satisfies $\tau(\Omega(\tau)) \subsetneqq \Omega(\tau)$;
\item
if the group $G$ is finite then there exists a pointwise nilpotent
cellular automaton $\tau \colon A^G \to A^G$ which is not nilpotent.
\end{enumerate}
\end{proposition}

\begin{proof}
Given a map  $f \colon A \to A$,
we consider the cellular automaton $\tau \colon A^G \to A^G$
with memory set $M \coloneqq \{1_G\}$ and associated local defining map
$\mu \coloneqq f  \colon A = A^M \to A$, that is, $\tau = \prod_{g \in G} f$.
\par
By Lemma~\ref{l:non-nil-limit-one-pt}.(i)
there exists $f \colon A \to A$ whose limit set is empty.
Clearly,  the associated cellular automaton $\tau \colon  A^G \to A^G$ has also empty limit set, showing (i).
\par
By Lemma~\ref{l:non-nil-limit-one-pt}.(ii),
there exists  a non-nilpotent map $f \colon A \to A$ such that $\Omega(f) = \{a_0\}$ for some  $a_0 \in A$.
Then, for such a choice of $f$, the cellular automaton   $\tau \colon A^G \to A^G$ is not nilpotent
and $\Omega(\tau) = \{x_0\}$, where $x_0 \in A^G$ is the constant configuration
defined by $x_0(g) \coloneqq  a_0$ for all $g \in G$. This shows (ii). 
\par
By Lemma~\ref{l:non-nil-limit-one-pt}.(iii),
we can find a map $f \colon A \to A$ which satisfies $f(\Omega(f)) \subsetneqq \Omega(f)$. 
Then, for such a choice of $f$, the cellular automaton   $\tau \colon A^G \to A^G$ clearly satisfies
$\tau(\Omega(\tau)) \subsetneqq \Omega(\tau)$.
This shows (iii).
\par
Finally, by Lemma~\ref{l:non-nil-limit-one-pt}.(iv),
there exists a surjective map $f \colon A \to A$ which is pointwise nilpotent.
The associated cellular automaton $\tau \colon A^G \to A^G$ is surjective and hence not nilpotent.
For $G$ finite, $\tau$  is clearly pointwise nilpotent. This shows (iv).  
\end{proof}

\subsection{Nilpotency and pointwise nilpotency of general cellular automata}

\begin{lemma}
\label{l:subshift-with-interior-is-full}
Let $A$ be a set and let $G$ be a group. 
Let $\Sigma \subset A^G$ be a topologically transitive  closed subshift. 
Suppose that $X \subset \Sigma$ is a closed subshift of $A^G$ with nonempty interior in $\Sigma$.
Then one has $X = \Sigma$.
\end{lemma}

\begin{proof} 
Let $U \subset \Sigma$ be a nonempty open subset of $\Sigma$. 
Let $V$ denote the interior of $X$ in $\Sigma$. 
Note that $V$ is $G$-invariant. 
By topological transitivity, 
there exists $g \in G$ such that
$U \cap gV \neq \varnothing$.
As $U \cap g V = U \cap V \subset U \cap X$, 
we deduce that $U \cap X \neq \varnothing$. 
Hence $X$ is dense in $\Sigma$. 
Since $X$ is also closed in $\Sigma$, we conclude that $X= \Sigma$. 
\end{proof}
 
\begin{lemma}
\label{l:weak-implies-nil}
Let $A$ be a set and let $G$ be an infinite  group. 
Let $\Sigma\subset A^G$ be a topologically mixing closed subshift of sub-finite-type. 
Suppose that $\tau \colon \Sigma \to \Sigma$ is a cellular automaton satisfying the following property:
there exists a constant configuration  $x_0 \in \Sigma$ 
such that, for every $x \in \Sigma$, there is an integer $n \geq 1$ such that $\tau^n(x) = x_0$.
Then $\tau$ is nilpotent with terminal point $x_0$.
\end{lemma}

\begin{proof}
Suppose first that $G$ is countable.
As $A^G$ is a countable product of discrete spaces, it admits a complete metric compatible with its topology.
Since $\Sigma$ is closed in $A^G$, 
it follows that the topology induced on  $\Sigma$ is completely metrizable 
 and hence that $\Sigma$ is a Baire space. For each integer $n \geq 1$,  the set
\[ 
X_n \coloneqq (\tau^n)^{-1}(x_0) = \{x \in \Sigma : \tau^n(x) = x_0\}  
\]
 is a closed subshift of $A^G$.
We have $\Sigma = \bigcup_{n \geq 1} X_n$ by our hypothesis on $\tau$.
By the Baire category theorem,
there is an integer $n_0 \geq 1$ such that $X_{n_0}$ has a nonempty interior.
The subshift $\Sigma$   is topologically mixing and therefore topologically transitive since $G$ is infinite.  
It follows that  $X_{n_0} = \Sigma$ by Lemma~\ref{l:subshift-with-interior-is-full}.
 Thus $\tau^{n_0}(x) = x_0$ for all $x \in \Sigma$.
 This shows that $\tau$ is nilpotent with terminal point $x_0$.
 Note that we have  not used the hypothesis that $\Sigma$ is of sub-finite-type in this part of the proof. 
 \par
 Let us treat now the general case.
 Suppose that $G$ is an infinite (possibly uncountable) group. 
 Let $M \subset G$ be a finite memory set for both $\tau$ and $\Sigma$. 
 As $G$ is infinite, there exists an infinite countable subgroup $H \subset G$ containing $M$.  
Let $\tau_H \colon \Sigma_H \to \Sigma_H$ denote the restriction cellular automaton (cf.~Section~\ref{s:restriction-ls}).
Thanks to the decompositions $\tau = \prod_{c \in G/H} \tau_c$ and $\Sigma = \prod_{c \in G/H} \Sigma_c$ 
where $\tau_c \colon \Sigma_c \to \Sigma_c$ (cf.~Section~\ref{s:restriction-ls}), 
it is not hard to see that $\Sigma_H$ and $\tau_H$ satisfy similar hypotheses as 
$\Sigma$ and $\tau$ 
with the constant terminal point $x_0\vert_H$. 
Remark that $\Sigma_H$ is topologically mixing since $H$ is infinite and 
$\Sigma$ is topologically mixing.  
Hence, $\tau_H$ is nilpotent by the above paragraph.  
Therefore, $\tau$ is itself nilpotent by Lemma~\ref{l:restriction-ls}.(ii).  
\end{proof}

The following result is well known,  at least in the case of full shifts with finite alphabets 
(cf.~\cite[Proposition~2]{guillon-richard-2008},  \cite[Proposition~1]{salo-nilpotent-2012}, \cite{meyerovitch-salo-2019}). 

\begin{proposition}
\label{p:carct-nilp-ca} 
Let $A$ be a set and let $G$ be an infinite group. 
Let $\Sigma \subset A^G$ be a topologically mixing closed subshift of sub-finite-type. 
Suppose that $\tau \colon \Sigma \to \Sigma$ is a cellular automaton.
Then the following conditions are equivalent:
\begin{enumerate}[\rm(i)]
\item
$\tau$ is nilpotent;
\item
$\tau$ is pointwise nilpotent;
\item
there exists a constant configuration  $x_0 \in \Sigma$ such that, for every $x \in \Sigma$, 
there is an integer $n \geq 1$ such that $\tau^n(x) = x_0$. 
\end{enumerate}\end{proposition}

\begin{proof}
The implication (i)$\implies$(ii) is obvious and (ii)$\implies$(iii)
immediately follows from $G$-equivariance of $\tau$. 
The implication (iii)$\implies$(i) follows from Lemma~\ref{l:weak-implies-nil}.
\end{proof} 

\begin{remark}
The equivalences (i)$\iff$(ii)$\iff$(iii) hold trivially true when $A$ and $G$ are both finite.
The implication (i)$\implies$(ii) and the equivalence (ii)$\iff$(iii) remain valid for $G$ finite.
However, it follows from Proposition~\ref{p:ca-non-nil-limit-pt}.(iv) that the implication (ii)$\implies$(i) 
becomes false for $A$ infinite and $G$ finite.
\end{remark}

\bibliographystyle{siam}

\begin{thebibliography}{10}

\bibitem{A-L}
{\sc S.~l.~O. Aanderaa and H.~R. Lewis}, {\em Linear sampling and the {$\forall
  \exists \forall $} case of the decision problem}, J. Symbolic Logic, 39
  (1974), pp.~519--548.

\bibitem{aubrun-barbieri-sablik-2017}
{\sc N.~Aubrun, S.~Barbieri, and M.~Sablik}, {\em A notion of effectiveness for
  subshifts on finitely generated groups}, Theoret. Comput. Sci., 661 (2017),
  pp.~35--55.

\bibitem{ax-injective}
{\sc J.~Ax}, {\em Injective endomorphisms of varieties and schemes}, Pacific J.
  Math., 31 (1969), pp.~1--7.

\bibitem{ballier-phd-2009}
{\sc A.~Ballier}, {\em Propri\'et\'e structurelle, combinatoires et logiques
  des pavages}, Ph.D. thesis. Aix-Marseille Universit\'e,  (2009).

\bibitem{ballier-durand-jeandel-2008}
{\sc A.~Ballier, B.~Durand, and E.~Jeandel}, {\em Structural aspects of
  tilings}, in S{TACS} 2008: 25th {I}nternational {S}ymposium on {T}heoretical
  {A}spects of {C}omputer {S}cience, vol.~1 of LIPIcs. Leibniz Int. Proc.
  Inform., Schloss Dagstuhl. Leibniz-Zent. Inform., Wadern, 2008, pp.~61--72.

\bibitem{Bar-Salo}
{\sc L.~Bartholdi and V.~Salo}, {\em Simulations and the lamplighter group}, to appear in
  Groups, Geometry, and Dynamics. arXiv:2010.14299.

\bibitem{berger-book-1966}
{\sc R.~Berger}, {\em The undecidability of the domino problem}, Mem.\ Amer.\  Math.\ Soc., 66 (1966), p.~72.

\bibitem{boyle-buzzi-gomez-2006}
{\sc M.~Boyle, J.~Buzzi, and R.~G\'{o}mez}, {\em Almost isomorphism for
  countable state {M}arkov shifts}, J. Reine Angew. Math., 592 (2006),
  pp.~23--47.

\bibitem{book}
{\sc T.~Ceccherini-Silberstein and M.~Coornaert}, {\em Cellular automata and
  groups}, Springer Monographs in Mathematics, Springer-Verlag, Berlin, 2010.

\bibitem{csc-algebraic}
\leavevmode\vrule height 2pt depth -1.6pt width 23pt, {\em On algebraic
  cellular automata}, J. Lond. Math. Soc. (2), 84 (2011), pp.~541--558.

\bibitem{csc-cat}
\leavevmode\vrule height 2pt depth -1.6pt width 23pt, {\em Surjunctivity and
  reversibility of cellular automata over concrete categories}, in Trends in
  harmonic analysis, vol.~3 of Springer INdAM Ser., Springer, Milan, 2013,
  pp.~91--133.

\bibitem{ccp-2019}
{\sc T.~Ceccherini-Silberstein, M.~Coornaert, and X.~K. Phung}, {\em On
  injective endomorphisms of symbolic schemes}, Comm. Algebra, 47 (2019),
  pp.~4824--4852.

\bibitem{ccp-goe-2020}
\leavevmode\vrule height 2pt depth -1.6pt width 23pt, {\em On the {G}arden of
  {E}den theorem for endomorphisms of symbolic algebraic varieties}, Pacific J.
  Math., 306 (2020), pp.~31--66.

\bibitem{JPAA}
\leavevmode\vrule height 2pt depth -1.6pt width 23pt, {\em On linear
  shifts of finite type and their endomorphisms}, J. Pure Appl. Algebra, 226
  (2022), pp.~Paper No. 106962, 27 pp.

\bibitem{culik-limit-sets-1989}
{\sc K.~Culik, II, J.~Pachl, and S.~Yu}, {\em On the limit sets of cellular
  automata}, SIAM J. Comput., 18 (1989), pp.~831--842.

\bibitem{cyr-kra-space-2019}
{\sc V.~Cyr, J.~Franks, and B.~Kra}, {\em The spacetime of a shift
  endomorphism}, Trans. Amer. Math. Soc., 371 (2019), pp.~461--488.

\bibitem{gromov-esav}
{\sc M.~Gromov}, {\em Endomorphisms of symbolic algebraic varieties}, J. Eur.
  Math. Soc. (JEMS), 1 (1999), pp.~109--197.

\bibitem{ega-1}
{\sc A.~Grothendieck}, {\em \'{E}l\'{e}ments de g\'{e}om\'{e}trie
  alg\'{e}brique. {I}. {L}e langage des sch\'{e}mas}, Inst. Hautes \'{E}tudes
  Sci. Publ. Math.,  (1960), p.~228.

\bibitem{grothendieck-20-1964}
\leavevmode\vrule height 2pt depth -1.6pt width 23pt, {\em \'{E}l\'ements de
  g\'eom\'etrie alg\'ebrique. {IV}. \'{E}tude locale des sch\'emas et des
  morphismes de sch\'emas. {I}}, Inst. Hautes \'Etudes Sci. Publ. Math.,
  (1964), p.~259.

\bibitem{grothendieck-4-3}
\leavevmode\vrule height 2pt depth -1.6pt width 23pt, {\em \'{E}l\'ements de
  g\'eom\'etrie alg\'ebrique. {IV}. \'{E}tude locale des sch\'emas et des
  morphismes de sch\'emas. {III}}, Inst. Hautes \'Etudes Sci. Publ. Math.,
  (1966), p.~255.

\bibitem{guillon-richard-2008}
{\sc P.~Guillon and G.~Richard}, {\em Nilpotency and limit sets of cellular
  automata}, in Mathematical foundations of computer science 2008, vol.~5162 of
  Lecture Notes in Comput. Sci., Springer, Berlin, 2008, pp.~375--386.

\bibitem{kari-nilpotency-1992}
{\sc J.~Kari}, {\em The nilpotency problem of one-dimensional cellular
  automata}, SIAM J. Comput., 21 (1992), pp.~571--586.

\bibitem{kitchens-book}
{\sc B.~P. Kitchens}, {\em Symbolic dynamics}, Universitext, Springer-Verlag,
  Berlin, 1998.
\newblock One-sided, two-sided and countable state Markov shifts.

\bibitem{lima-sarig-2019}
{\sc Y.~Lima and M.~Sarig}, {\em Symbolic dynamics for three-dimensional flows
  with positive topological entropy}, J. Eur. Math. Soc. (JEMS), 21 (2019),
  pp.~199--256.

\bibitem{lind-marcus}
{\sc D.~Lind and B.~Marcus}, {\em An introduction to symbolic dynamics and
  coding}, Cambridge University Press, Cambridge, 1995.

\bibitem{liu-book}
{\sc Q.~Liu}, {\em Algebraic geometry and arithmetic curves}, vol.~6 of Oxford
  Graduate Texts in Mathematics, Oxford University Press, Oxford, 2002.
\newblock Translated from the French by Reinie Ern\'{e}, Oxford Science
  Publications.

\bibitem{meyerovitch-salo-2019}
{\sc T.~Meyerovitch and V.~Salo}, {\em On pointwise periodicity in tilings,
  cellular automata, and subshifts}, Groups Geom. Dyn., 13 (2019),
  pp.~549--578.

\bibitem{milne-group-book}
{\sc J.~S. Milne}, {\em Algebraic groups}, vol.~170 of Cambridge Studies in
  Advanced Mathematics, Cambridge University Press, Cambridge, 2017.
\newblock The theory of group schemes of finite type over a field.

\bibitem{milnor-ca-1988}
{\sc J.~Milnor}, {\em On the entropy geometry of cellular automata}, Complex
  Systems, 2 (1988), pp.~357--385.

\bibitem{moore}
{\sc E.~F. Moore}, {\em Machine models of self-reproduction}, vol.~14 of Proc.
  Symp. Appl. Math., American Mathematical Society, Providence, 1963,
  pp.~17--34.

\bibitem{myhill}
{\sc J.~Myhill}, {\em The converse of {M}oore's {G}arden-of-{E}den theorem},
  Proc. Amer. Math. Soc., 14 (1963), pp.~685--686.
	
\bibitem{neumann-book}
{\sc J.~von Neumann}, {\em Theory of self-reproducing automata}, Univerity of
  Illinois Press, 1966. \newblock (A.W.\ Burks, ed.).

\bibitem{osipenko-2007}
{\sc G.~Osipenko}, {\em Dynamical systems, graphs, and algorithms}, vol.~1889
  of Lecture Notes in Mathematics, Springer-Verlag, Berlin, 2007.
\newblock Appendix A by N. B. Ampilova and Appendix B by Danny Fundinger.

\bibitem{phung-group-2020}
{\sc X.~K. Phung}, {\em On dynamical finiteness properties of algebraic group
  shifts}, to appear in Israel Journal of Mathematics. arXiv:2010.04035.

\bibitem{phung-2018}
\leavevmode\vrule height 2pt depth -1.6pt width 23pt, {\em On sofic groups,
  {K}aplansky's conjectures, and endomorphisms of pro-algebraic groups}, J.
  Algebra, 562 (2020), pp.~537--586.
	
\bibitem{phung-dcds}
\leavevmode\vrule height 2pt depth -1.6pt width 30pt,
{\em Shadowing for families of endomorphisms of generalized group shifts},  
Discrete \& Continuous Dynamical Systems, 2022, 42 (1): pp.~285--299.

\bibitem{phung-post-surjective}
\leavevmode\vrule height 2pt depth -1.6pt width 30pt,
{\em On symbolic group varieties and dual surjunctivity}, 
to appear in Groups, Geometry, and Dynamics. arXiv:2111.02588 
 

\bibitem{robinson-1971}
{\sc R.~Robinson}, {\em Undecidability and nonperiodicity for tilings of the
  plane}, Invent. Math., 12 (1971), pp.~177--209.

\bibitem{salo-nilpotent-2012}
{\sc V.~{Salo}}, {\em {On Nilpotency and Asymptotic Nilpotency of Cellular
  Automata}}, in Cellular Automata and Discrete Complex Systems and 3rd
  international symposium Journ\'ees Automates Cellulaires, AUTOMATA \& JAC
  2012, La Marana, Corsica, 2012, pp.~86--96.

\bibitem{salo-nil-2017}
{\sc V.~Salo}, {\em Strict asymptotic nilpotency in cellular automata}, in
  Cellular automata and discrete complex systems, vol.~10248 of Lecture Notes
  in Comput. Sci., Springer, Cham, 2017, pp.~3--15.

\bibitem{sarig-1999}
{\sc M.~Sarig}, {\em Thermodynamic formalism for countable {M}arkov shifts},
  Ergodic Theory Dynam. Systems, 19 (1999), pp.~1565--1593.

\bibitem{sarig-2013}
\leavevmode\vrule height 2pt depth -1.6pt width 23pt, {\em Symbolic dynamics
  for surface diffeomorphisms with positive entropy}, J. Amer. Math. Soc., 26
  (2013), pp.~341--426.

\bibitem{shub-global-stability}
{\sc M.~Shub}, {\em Global stability of dynamical systems}, Springer-Verlag,
  New York, 1987.
\newblock With the collaboration of A. Fathi and R. Langevin, Translated from
  the French by Joseph Christy.

\bibitem{vakil}
{\sc R.~Vakil}, {\em {\itshape MATH 216: Foundations of Algebraic Geometry}}.

\bibitem{wolfram-univ-1984}
{\sc S.~Wolfram}, {\em Universality and complexity in cellular automata},
  vol.~10, 1984, pp.~1--35.
\newblock Cellular automata (Los Alamos, N.M., 1983).

\bibitem{wolfram-new-kind}
\leavevmode\vrule height 2pt depth -1.6pt width 23pt, {\em A new kind of
  science}, Wolfram Media, Inc., Champaign, IL, 2002.

\end{thebibliography}
\def\cprime{$'$}

\end{document}